\documentclass[11pt]{amsart}

\usepackage{url}
\usepackage{mathdots}
\usepackage{amscd,latexsym,amsthm,amsfonts,amssymb,amsmath,amsxtra,hyperref,xcolor}
\usepackage[mathscr]{eucal}
\usepackage{mathabx}
\usepackage[toc,page,header]{appendix}
\usepackage [english]{babel}
\usepackage [autostyle, english = american]{csquotes}
\MakeOuterQuote{"}
\usepackage{array}
\usepackage{xr-hyper}
\usepackage{hyperref} 
\usepackage{lipsum}

\usepackage[all]{xy}

\newcommand{\BC}{{\mathbb {C}}}

\newcommand{\BE}{{\mathbb {E}}}
\newcommand{\BF}{{\mathbb {F}}}

\newcommand{\BN}{{\mathbb {N}}}

\newcommand{\BQ}{{\mathbb {Q}}}
\newcommand{\BR}{{\mathbb {R}}}

\newcommand{\BZ}{{\mathbb {Z}}}

\newcommand{\bJ}{{\mathbf {J}}}

\newcommand{\CB}{{\mathcal {B}}}
\newcommand{\CC}{{\mathcal {C}}}

\newcommand{\CF}{{\mathcal {F}}}
\newcommand{\CG}{{\mathcal {G}}}

\newcommand{\CJ}{{\mathcal {J}}}
\newcommand{\CK}{{\mathcal {K}}}

\newcommand{\CM}{{\mathcal {M}}}
\newcommand{\CN}{{\mathcal {N}}}

\newcommand{\CT}{{\mathcal {T}}}
\newcommand{\CU}{{\mathcal {U}}}

\newcommand{\CW}{{\mathcal {W}}}

\newcommand{\CZ}{{\mathcal {Z}}}

\newcommand{\FA}{{\mathfrak {A}}}

\newcommand{\FF}{{\mathfrak {F}}}

\newcommand{\FO}{{\mathfrak {O}}}
\newcommand{\FP}{{\mathfrak {P}}}

\newcommand{\Fk}{{\mathfrak {k}}}

\newcommand{\Fo}{{\mathfrak {o}}}
\newcommand{\Fp}{{\mathfrak {p}}}

\newcommand{\Vol}{{\mathrm {Vol}}}

\newcommand{\RB}{{\mathrm {B}}}

\newcommand{\RI}{{\mathrm {I}}}

\newcommand{\RK}{{\mathrm {K}}}

\newcommand{\RM}{{\mathrm {M}}}

\newcommand{\RP}{{\mathrm {P}}}
\newcommand{\RQ}{{\mathrm {Q}}}

\newcommand{\RU}{{\mathrm {U}}}

\newcommand{\RZ}{{\mathrm {Z}}}

\newcommand{\End}{{\mathrm{End}}}

\newcommand\Frob{\textup{Frob}}

\newcommand\Gal{\textup{Gal}}
\newcommand{\GL}{{\mathrm{GL}}}

\newcommand{\Hom}{{\mathrm{Hom}}}

\newcommand{\Ind}{{\textup{Ind}}}
\newcommand{\cInd}{{\mathrm{c\text{-}Ind}}}

\newcommand{\I}{{\mathrm{I}}}

\newcommand{\Ker}{{\mathrm{Ker}}}

\newcommand{\ord}{{\mathrm{ord}}}

\renewcommand{\Re}{{\mathrm{Re}}}

\newcommand{\Tr}{{\mathrm{Tr}}}

\newcommand{\wh}{\widehat}

\def\sig{{\sigma}}

\newtheorem*{thm*}{Theorem}
\newtheorem{thm}{Theorem}[section]
\newtheorem{cor}[thm]{Corollary}
\newtheorem{lem}[thm]{Lemma}
\newtheorem{lm}[thm]{Lemma}
\newtheorem{prop}[thm]{Proposition}

\newtheorem{conj}[thm]{Conjecture}

\newtheorem{ques/conj}[thm]{Question/Conjecture}
\newtheorem{ques}[thm]{Question}
\newtheorem{defn}[thm]{Definition}
\newtheorem{rmk}[thm]{Remark}

\newtheorem{exmp}[thm]{Example}

\makeatletter

\newcommand{\Rmnum}[1]{\expandafter\@slowromancap\romannumeral #1@}
\makeatother

  \newcommand{\blue}[1]{\textcolor{blue}{{#1}}}
\newcommand{\Nr}{{\mathrm{Nr}}}
\newcommand{\apair}[1]{\left\langle {#1} \right\rangle}

\newcommand{\ov}{\overline}
\newcommand{\Irr}{{\mathrm{Irr}}}

 \input xy
  \xyoption{all}

\usepackage{hyperref}

\numberwithin{equation}{section}

\def\AA{\mathbb{A}}

\def\CC{\mathbb{C}}

\def\FF{\mathbb{F}}
\def\GG{\mathbb{G}}

\def\QQ{\mathbb{Q}}

\def\ZZ{\mathbb{Z}}

\newcommand\cK{\mathcal{K}}
\newcommand\cL{\mathcal{L}}

\def\bJ{\mathbf{J}}

\def\bR{\mathbf{R}}

\newcommand\AS{\textup{AS}}

\newcommand{\Res}{\textup{Res}}

\newcommand\Spec{\textup{Spec}\ }

\newcommand\uEnd{\underline{\End}}
\newcommand{\Nm}{\textup{Nm}}
\newcommand{\Gm}{\GG_m}

\newcommand{\incl}{\hookrightarrow}

\newcommand{\Qlbar}{\overline{\QQ}_\ell}

\newcommand\ot{\otimes}

\newcommand{\cohog}[2]{\textup{H}^{#1}({#2})}     
\newcommand{\cohoc}[2]{\textup{H}_{c}^{#1}({#2})}     

\renewcommand\a\alpha
\renewcommand\b\beta
\newcommand\g\gamma
\renewcommand\d\delta

\newcommand{\s}{\sigma}

\newcommand{\ep}{\epsilon}

\newcommand{\kbar}{\overline{k}}

\newcommand{\Kl}{\textup{Kl}}
\newcommand{\Swan}{\textup{Swan}}

\title{Converse Theorem Meets Gauss Sums}

\author{Chufeng Nien}
\address{Department of Mathematics, National Cheng Kung University, Tainan 701, Taiwan}
\email{nienpig@mail.ncku.edu.tw}

\author{Lei Zhang \\
(with an appendix by Zhiwei Yun)}
\address{Department of Mathematics,
National University of Singapore,
Singapore 119076}
\email{matzhlei@nus.edu.sg}

\address{Department of Mathematics, Massachusetts Institute of Technology, 77 Massachusetts Ave, Cambridge, MA 02139}
\email{zyun@mit.edu}

\date{}

\keywords{ Whittaker models, Gamma factors,  Gauss sums, $p$-adic Gamma function, Kloosterman sheaves}

\begin{document}

\begin{abstract}
This paper verifies $n\times 1$ Local Converse Theorem for twisted gamma factors of irreducible cuspidal representations of $\GL_n(\BF_p)$, for $n\leq 5,$ and  of irreducible generic representations, for $n<\frac{q-1}{2\sqrt{q}}+1$ in the appendix by Zhiwei Yun, where $p$ is a prime and q is a power of $p$. 
 The counterpart of $n\times 1$ converse theorem for level zero cuspidal representations also follows the established relation between gamma factors of $\GL_n(\CF)$ and that  of $\GL_n(\BF_q)$, where $\CF$ denotes a $p$-adic field whose residue field is isomorphic to $\BF_q.$
For $n=6,$ examples failed $n\times 1$ Local Converse Theorem over finite fields are provided and the authors propose a set of primitive representations, for which $n\times 1$ gamma factors should be able to detect a unique element in it.

For $m,\ n\in \BN,$ in the spirit  of Langlands functorial lifting, we formulate a conjecture to relate $n\times m$ gamma factors of finite fields with Gauss sums over extended fields.   
\end{abstract}

\thanks{The work of the first named author is supported by  Ministry of Science and Technology, Taiwan 104-2115-M-006-008-,
and that of the second named author is supported in part by AcRF Tier 1 grant R-146-000-237-114 of National University of Singapore.
2010 Mathematics Subject Classification. Primary 11T24; Secondary
11Z05.
} 
\maketitle

\section{Motivation}

Let $F$ be either a $p$-adic field $\CF$ or a finite field $\BF_q$.
In the celebrated paper \cite{JP-SS83}, Jacquet, Piatetski-Shapiro, and Shalika 
defined the  local gamma factors $\Gamma(s,\pi\times\tau,\psi)$ through the functional equation of Rankin-Selberg integrals for a pair of irreducible generic representations, $\pi$ of $\GL_n(\CF)$ and $\tau$ of $\GL_t(\CF).$   
Then a natural question in representation theory arises: "How much invariants of a cuspidal (or generic) 
representation can be determined in terms of its twisted gamma factors?"
 \begin{ques}[$n\times m$ Local Converse Problem]\label{que:converse}
Let $\pi_1$ and $\pi_2$ be irreducible (super)cuspidal representations representations of $\GL_n(F)$.
Suppose that the gamma factors $\Gamma(s,\pi_1\times\tau,\psi)=\Gamma(s,\pi_2\times\tau,\psi)$ for all  $\tau\in\CT(m)$,
where   $m\leq n$ and
$$\CT(m)=\coprod_{1\leq d\leq m}\{\tau\colon \text{ $\tau$ is an irreducible generic representation of $\GL_d(F)$}\}.$$ 
Can we deduce that $\pi_1\cong \pi_2$? 
\end{ques}
 In 1993, Henniart answered Question \ref{que:converse} with $n\times (n-1)$ Local Converse Theorem in \cite{He93}.
 In 1996, Jeff Chen proved $n\times (n-2)$ Local Converse Theorem in \cite{Ch96}. 
However, a conjecture credited to Jacquet predicted that $m=[\frac{n}{2}]$ should be enough to confirm 
Question \ref{que:converse} positively. 
For a while, no further  progress has been made towards Jacquet's conjecture on Local Converse Problem and eventually people try to look for inspiration from finite field case. 

Roditty in her thesis \cite{Ro10}, under the supervision of Soudry, considered the finite field analogue of gamma factors $\gamma( \pi_1\times\tau,\psi)$ for an irreducible cuspidal represntation of $\GL_n(\BF_q)$ and 
an irreducible generic represntation $\tau$ of $\GL_n(\BF_q),$ for $n>t.$ 
She also verified $n\times (n-1)$ and $n\times (n-2)$ Local Converse Theorems for $\GL_n(\BF_q).$ 
In 2014, the first named author verified the finite field analogue of Jacquet's conjecture as follows. 
\begin{thm}[{$n\times [\frac{n}{2}]$ Local Converse Theorem for finite field case}, \cite{Ni14}]\label{finiteJ}
Let $\pi_1$ and $\pi_2$ be two irreducible cuspidal
representations of $\GL_n(\BF_q)$ with the same central character. 
If $$\gamma(\pi_1\times \tau,\psi)=\gamma(\pi_2\times \tau,\psi),$$ for all irreducible generic representation of
$\GL_r(\BF_q),$ $1\le r \le [\frac{n}{2}],$ then $\pi_1\cong \pi_2.$
\end{thm}
In \cite{JiNiS15}, Nien, Jiang and Stevens  
set up a formulation for verifying Jacquet's conjecture 
  and use this formulation to confirm many cases.
Following the formulation,  Adrian, Liu, Stevens, and Xu (\cite{ALSX16}) verified Jacquet's conjecture for $\GL_p(\CF)$ for prime $p$.
Finally,
 Jacquet and Liu (\cite{JL16}), using analytic approaches,
and Chai (\cite{Cha16}), using $p$-adic Bessel functions, independently
 settled this long standing conjecture.

Next,  we may wonder if the bound $[\frac{n}{2}]$ is sharp for Question \ref{que:converse}. 
In \cite{ALST17}, Adrian, Liu, Stevens, and Tam 
 showed that $[\frac{n}{2}]$ is the sharp bound for necessary twisting in  Local Converse Theorem for a pair of supercuspidal representations of   $\GL_n(\CF),$ when $n$ is a prime.

Inspired by \cite[Conjecture 2]{P-S01}, Piatetski-Shapiro's speculation on the $\GL_1$-twisted global converse problem, 
  rather than  looking for a minimal set of
invariants to determine a representation uniquely (up to equivalence), we are interested in how much information of representations carried by $n\times 1$ gamma factors. 
Hence we recast Local Converse Problems as follows. 
\begin{ques}\label{que:converse-finite}
Does there exist a (possibly maximal)subset $\Omega$ of irreducible generic representations, whose elements can be uniquely characterized by $n\times 1$ gamma factors? 
More precisely,  we desire a canonical set  $\Omega$ so that for given $\pi_1,\pi_2\in \Omega,$
 if $\gamma(s,\pi_1\times\omega,\psi)=\gamma(s,\pi_2\times\omega,\psi)$\footnote{When $F$ is finite, a gamma function is merely a complex number independent of the parameter $s$. Here we abuse by the notation $\gamma(s,\pi_1\times\omega,\psi)$ to indicate the similarity of related questions or conjectures over $p$-adic fields and over finite fields.}
 for all characters $\omega$ of $\GL_1(F),$ then $\pi_1\cong \pi_2$.
\end{ques}
In \cite{BuH14},  Bushnell and Henniart answered the above question with the set of
simple cuspidal representations (i.e. supercuspidal representations with
minimal positive depth) of $\GL_n(\CF)$ over $p$-adic fields. 

In this paper, we focus on the cases for the level zero supercuspidal representations and its counterpart over finite fields.
Different from those analytic technique and algebraic approaches applied on Local Converse Theorem,  we appeal to the arithmetic natural encoded in gamma factors. 
The (abelian) Gauss sums involved in the gamma factors over finite fields are algebraic integers, and 
  the fruitful results developed for (abelian) Gauss sums enable us to extract the invariants for characterizing our target representations. 

\section{Introduction}

First, we recall Green's work on irreducible (trace) characters of $\GL_n(\BF_q)$.
Throughout this paper, let $p$ denote a prime and $\BF_q$ be a finite field of cardinality $q$. 
\begin{defn}\label{regular}
A character of $\BF_{q^n}^\times$  is called {\bf regular} if it does not factor through  $\Nr_{n:d}$ for all divisors $d$ of $n$ and $1\le d<n$, where $\Nr_{n:d}$ is the reduce norm map from $\BF_{q^n}\mapsto \BF_{q^d}$. 

Two regular characters $\chi_1,\  \chi_2$ are called {\bf equivalent }
if $\chi_1=\chi_2^{q^j}$ for some integer $j$.
\end{defn}
In \cite{Gre55},  Green constructed all irreducible trace characters 
of $\GL_n(\BF_q)$ in terms of associated characters of
$$\BF_{q^{n_1}}^\times \times \BF_{q^{n_2}}^\times \times \cdots \times \BF_{q^{n_k}}^\times,
\text{ where }n_1+\cdots+n_k=n.$$ 
Moreover, 
cuspidal representations of $\GL_n(\BF_q)$ are one-to-one corresponding to the equivalence classes of the regular characters of $\BF_{q^n}^\times$, which can be viewed as the Langlands correspondence for the finite field case. 
Remark that Springer correspondence can completely parameterize all irreducible representations of $\GL_n(\BF_q)$ via Deligne-Lusztig characters. See \cite{Ca} for instance. 

Guided by the hypothesis of Langlands functoriality, there should exist a tensor lifting from $\pi\times\tau$ of $\GL_n(\BF_q)\times\GL_m(\BF_q)$ to $\Pi$ of $\GL_{mn}(\BF_q)$ such that 
$\Pi$ corresponds to the character $\chi\circ \Nr_{mn:n}\cdot \eta\circ \Nr_{mn:m}$ of $\BF^\times_{q^{mn}}$.
Note that the character $\chi\circ \Nr_{mn:n}\cdot \eta\circ \Nr_{mn:m}$ is possibly not regular and then the corresponding representation $\Pi$ is not cuspidal. 
We expect that $\Pi$ corresponds to the unique generic subrepresentation of $$\Ind_{\GL_{n_1}\times\cdots\times\GL_{n_k}}^{\GL_{mn}}\pi_1\times\cdots\times \pi_k,$$   
where $n_1+\cdots+n_k=mn$ and $\pi_i$ is a cuspidal representation of $\GL_{n_i}.$
If we extend Roditty's gamma factors to generic ones through the multiplicativity\footnote{The multiplicativity of gamma factors of $\GL_n(\BF_q)$ has been established in Soudry's unpublished note.}, then 
$\gamma(\Pi\times 1,\psi):=\prod_{i=1}^k\gamma(\pi_i\times 1,\psi)$ and we expect
$\gamma(\Pi\times 1,\psi)=q^{\frac{2-m^2-m}{2}}\gamma(\pi\times\tau,\psi).$  
To answer Question \ref{que:converse-finite}, based on the known case $m=1$, we proposed the following conjectural formula for gamma factors in terms of the corresponding characters of $\BF_{q^{mn}}^\times$.
\begin{conj}\label{conj:gamma}
For $n>m,$ 
let $\pi$ and $\tau$ be irreducible cuspidal representations of $\GL_n(\BF_q)$ and $\GL_m(\BF_q)$ respectively, and
$\chi$ and $\eta$ be the corresponding regular characters of $\BF_{q^n}^\times$ and of $\BF_{q^m}^\times$.
Then  
\begin{equation}\label{eq:conj-tensor-gauss}
\gamma(\pi\times\tau,\psi)
=c\cdot\chi(-1)^{m-1}\eta(-1)^{n-1}G(\chi\circ \Nr_{mn:n}\cdot \eta\circ \Nr_{mn:m},\psi),
\end{equation}
where    $$G(\beta,\psi):=\sum_{a\in\BF^{\times}_{q^{N}}}\beta(a)\psi(\Tr a^{-1})=\sum_{a\in\BF^{\times}_{q^{N}}}\beta^{-1}(a)\psi(\Tr a)$$ is the Gauss sum for a character $\beta^{-1}$ of $\BF_{q^{N}}^\times,$
  $\Tr$ is the reduced trace map from $\BF_{q^N}\mapsto \BF_q $ and $c=(-1)^{m(n-1)}q^{-mn+\frac{m^2+m}{2}}.$ 
\end{conj}
For instance, let $n=2m$ and $\chi_i$ be regular characters of $\BF_{q^{2m}}^\times$ where $i=1,\ 2$, and  $\pi_i$ be the cuspidal representations of $\GL_{2m}$ corresponding to $\chi_i$.
Presuming the validity of Conjecture \ref{conj:gamma}, one has 
$\gamma(\pi_1\times\tau,\psi)=\gamma(\pi_2\times\tau,\psi)$ is equivalent to $G(\chi_i\cdot\eta,\psi)^m=G(\chi_i\cdot\eta,\psi)^m$ for all cuspidal representations $\tau$ of $\GL_m(\BF_q)$,
where $\eta$ is the regular character associated to $\tau$.
Following our method in Section \ref{sec:main}, it is easy to check that if $G(\chi_1\cdot \eta,\psi)^m=G(\chi_2\cdot \eta,\psi)^m$ for all characters $\eta\in\hat{\BF}^\times_{q^m}$, then $\chi_1=\chi_2^{q^j},$
which implies the even case of Theorem \ref{finiteJ}.

The above conjecture interprets representation theoretic material in terms of number theoretic objects.
For the case $m=1$, this conjectural formula, to be recalled in Theorem \ref{gauss sum}, has been verified by 
the first named author in \cite{Ni17}.
 When $m$ divides $n$, Conjecture \ref{conj:gamma} and the multiplicativity of the gamma factors imply Hasse-Davenport relation for the Gauss sums.

Now, Question \ref{que:converse-finite} is reduced to the study of the twisted Gauss sums in 
Eq. \eqref{eq:conj-tensor-gauss}.
Gross-Koblitz's formula (to be introduced in section \ref{GK}) transforms the Gauss sums $G(\cdot,\psi)$ into more computable $p$-adic $\Gamma$-functions.
By using Gross-Koblitz's formula and Stickelberger's Theorem, we are able to track the values of $\eta$-twisted Gauss sums and obtain the following results for $\GL_n(\BF_p)$. 

\begin{thm}\label{thm:main-5} Let $p$ be a prime.
Suppose $\pi$ is an irreducible cuspidal representations of $\GL_n(\BF_p)$ and $n\leq 5$.
If $\gamma(\pi\times\eta,\psi)=\gamma(\tau\times\eta,\psi)$ for all characters $\eta$ of $\BF_p^\times$, then
$\pi\cong \tau$.
\end{thm}

Theorem \ref{thm:main-5} also holds for the  level zero supercuspidal representations by passing through     the relation, to be built in Theorem \ref{level0}, between the twisted gamma factors for a pair of  level zero supercuspidal representations $\pi\otimes\tau$ of $\GL_n(\CF)\times \GL_t(\CF)$ and the twisted gamma factors for a pair of  cuspidal representations $\pi_0\otimes\tau_0$ of $\GL_n(\BF_p)\times\GL_t(\BF_p)$, where $\pi$ (resp. $\tau$) is compactly induced from the inflation of $\pi_0$ (resp. $\tau_0)$. 

Furthermore, Proposition \ref{614} shows that the $\GL_1(\BF_q)$-twisted gamma factors can be characterized by the Mellin transforms of the Bessel functions $\CB_{\pi, \psi},$ defined in Proposition \ref{defnB}.
If those types of cuspidal representations are uniquely determined by their  $\GL_1(\BF_q)$-gamma factors,
the restrictions of $\CB_{\pi, \psi}$ into $\GL_1(\BF_q)$-(embedded in the left lower corner) are also unique. 
By Theorems \ref{thm:main-5} and \ref{614}, the following uniqueness property in terms of restricted values of Bessel functions holds. 
\begin{cor}\label{cor:Bessel-unique}
Suppose that $\pi$ and $\tau$ are irreducible cuspidal representations of $\GL_n(\BF_p)$ and $n\leq 5$.
If $$\CB_{\pi,\psi}(\begin{pmatrix}0&\RI_{n-1}\\
a&0\end{pmatrix})=\CB_{\tau,\psi}(\begin{pmatrix}0&\RI_{n-1}\\
a&0\end{pmatrix}),\text{ for all } a\in\BF_p^\times,$$ then
$\pi\cong \tau$.
\end{cor}

In addition, we apply this method to investigate certain cases in Sections \ref{sec:example-2} and \ref{sec:example-3} for some particular primes $p$ and higher ranks.
In the  inspiration of Arthur's work in \cite{Ar}, we introduce a {\bf primitive form} of cuspidal representations to formulate our conjecture for a possible answer to Question \ref{que:converse-finite}.
Let $r$ be a prime divisor of $n$. 
For each regular character $\chi$ of $\BF_{q^n}^\times$, denote by $\pi_\chi$ a cuspidal representation of $\GL_{r}(\BF_{q^{n/r}})$ corresponding to $\chi$, which is called a {\bf primitive representation} associated to  $\chi$.
\begin{conj}\label{conj:gamma-min-prime}\label{conj1} 
For $i=1,\ 2$, let $\chi_i$ be a regular character of $\BF_{q^n}^\times$
and $r$ be a prime divisor of $n$.
If $\gamma(\pi_{\chi_1}\times\eta,\psi)=\gamma(\pi_{\chi_2}\times\eta,\psi)$ for all characters $\eta$ of $\BF_{q^{n/r}}^\times$.
Then $\pi_{\chi_1}\cong \pi_{\chi_2}$.	
\end{conj} 

Note that in the above conjecture, we predict the validity of $n\times 1$  Local Converse Theorem  when $n$ is prime and $m=1$.
We rewrite it in terms of Gauss sums, as follows.
\begin{conj}\label{conj:Gauss-sum}
Let $\chi_1$ and $\chi_2$ be regular characters of $\BF_{q^n}^\times$ and $n$ be a prime. 
If $G(\chi_1\cdot\eta,\psi)=G(\chi_2\cdot\eta,\psi)$ for all $\eta\in\hat{\BF}^\times_q$,
then $\chi_1=\chi_2^{q^j}$ for some $j$.
\end{conj}

In the current paper, we verify Conjecture \ref{conj:Gauss-sum} for the cases $n\le 5,\ r=1,\ q$ is a prime in  Theorem  \ref{thm:main-5} and the cases when $q=2$ and  $2^n-1$ is a Mersenne prime in Section \ref{sec:example-2}.
In Section \ref{sec:example-3}, we use the results of pure Gauss sum to 
give counter-examples to  $n\times 1$ Local Converse Theorem, when $q=3$ and $n\geq 6$ is even. In fact, these counter-examples motivate the formulation of the above conjectures.
 
After a primitive version of this paper is sent to Zhiwei Yun, he showed us a parallel statement in terms of Kloosterman sheaves and proved very general results.
\begin{thm}[Theorem \ref{th:G} in Appendix]\label{YunApp}
Let $n\ge1$ be an integer satisfying $n<\frac{q-1}{2\sqrt{q}}+1$. Let $\chi_{1}, \ \chi_{2}:\BF_{q^{n}}^{\times}\to \BC^{\times}$ two characters. Suppose that
\begin{equation*}
G(\chi_{1}\cdot \eta, \psi)=G(\chi_{2}\cdot \eta, \psi),
\end{equation*}
for any character $\eta:\BF^{\times}_{q}\to \BC^{\times}.$
Then $\chi_{1}$ and $\chi_{2}$ are in the same Frobenius orbit, i.e., there exists $j\in\BZ$ such that $\chi_{1}=\chi_{2}^{q^{j}}$. 
\end{thm}	

More generally, referring to Remark \ref{rmk:appendix}, under the same assumption on $n$, Theorem \ref{th:G} holds for any character $\chi: A^{\times}\to \BC^{\times}$ of any \'etale $\BF_{q}$-algebra $A$ of  degree $n$. 
Such characters correspond to the generic representations $\sig_\chi$ of Levi subgroups of $\GL_n$.
Applying the multiplicativity of the gamma factors defined by Roditty, 
the Gauss sum $G_A(\chi,\psi)$ defined in Remark \ref{rmk:appendix} is the gamma factor of the generic subrepresentation of the induced representations from $\sig_\chi$.
Hence, it follows that Theorem \ref{thm:main-5} and Corollary \ref{cor:Bessel-unique} also hold for all irreducible generic representations when $n<\frac{q-1}{2\sqrt{q}}+1$.
Hence, the set of all irreducible generic representations is the desired one to
 answer Question \ref{que:converse-finite} when $n<\frac{q-1}{2\sqrt{q}}+1$.

{\it  Acknowledgment.}
We are grateful to Dihua Jiang for bringing this problem to our attention and many useful discussions. 
We would also like to thank David Soudry for his comments on the multiplicativity of gamma factors over the finite field case and for his other helpful comments.
We also appreciate the comments from David Soudry  about the multiplicativity of gamma factors over the finite field case.
Last, but not least, we would like to thank Zhiwei Yun for his interest in the first version of this manuscript,
which leads to his appendix revealing the geometric perspective on the related problems.


\section{Local Gamma factor and converse theorem}

\maketitle


 \subsection{Whittaker models}\label{2.1}
Let $F$ be either a finite field or a $p$-adic field.
In this section all groups are considered over $F$ unless state otherwise.
Let $\psi$ be a fixed nontrivial additive character of $F$. Let
$\RU_n$ be the unipotent radical of the standard Borel  subgroup
 $\RB_n$ of $\GL_n.$
Denote by $\RP_n$ the mirabolic subgroup of $\GL_n$, consisting of
matrices in $\GL_n$, whose last row equal to $(  0,\cdots, 0
,1)$. 
Let $\RZ_n$ denote the center of  $\GL_n$. 

\begin{defn}\label{nonde}
A character $\psi'$ of $\RU_n$ is called {\bf non-degenerate} if
 $$\psi' (u)=\psi(\sum_{i=1}^{n-1}a_iu_{i,i+1}),
  \text{ for }u=(u_{i,j})\in \RU_n,$$
  for some $a_i\in F^\times.$
Also, we denote by $\psi_n$, the  {\bf standard non-degenerate} character   given by
$$\psi_n (u)=\psi(\sum_{i=1}^{n-1} u_{i,i+1}),
  \text{ for }u=(u_{i,j})\in \RU_n.$$
\end{defn}

\begin{defn} Let $\pi$ be an irreducible representation of
 $\GL_n.$
Given a non-degenerate character $\psi'$ of $\RU_n,$ we call $\pi$
{\bf $\psi'$-generic} if $$\dim\Hom_{\GL_n}(\pi,
\Ind_{\RU_n}^{\GL_n}\psi')\ne 0.$$
\end{defn}
If $\pi$ is $\psi'$-generic, then it is also
$\psi''$-generic for any other non-degenerate character $\psi''$
of $\RU_n.$ Therefore, we will use the term "genericity"  instead of
"$\psi'$-genericity"
 from now on.
In the following, we recall some well known results about Whittaker models.

\begin{thm}
[Uniqueness of Whittaker model]\label{uni} Let $\pi$
be an irreducible representation of $\GL_n.$ Then
$$\dim\Hom_{\RU_n}(\pi|_{\RU_n}, \psi_n)=\dim\Hom_{\GL_n}(\pi, \Ind_{\RU_n}^{\GL_n}\psi_n)\le 1.$$
\end{thm}

When $\pi$ is generic, the above $\Hom$-space is of  dimension one. Let $\ell_{\psi_n}\in \Hom_{\RU_n }(\pi|_{\RU_n}, \psi_n)$ be a
nonzero Whittaker functional of $\pi$. Then for $v\in V_\pi$, define $W_v(g):=\ell_{\psi_n}(\pi(g)v)$, which is
called the Whittaker function attached to the vector $v$ and belongs to the induced representation
$\Ind_{\RU_n}^{\GL_n}\psi_n$.
By Theorem \ref{uni}, the subspace generated by all Whittaker functions $W_v(g)$ is
unique and  will be denoted by $\CW(\pi, \psi_n)$. This space is called the {\bf Whittaker model} of $\pi$.
It is known that any irreducible (super)cuspidal representation of $\GL_n(F)$ is generic.

\subsection{Gamma factors over finite fields}\label{2.2}

In \cite{Ro10}, Roditty considered a finite-field-analogue of Tate's thesis for the local functional equation and defined the gamma factor for a character of $\BF_q^\times$ as the proportional factor.
For $n>1,$  the functional equations for irreducible
cuspidal representations of $\GL_n(\BF_q)$ and irreducible generic representation $\tau$
 of $\GL_t(\BF_q)$ with $n>t$ are established in the following Theorem.

\begin{thm}[{\cite[Theorems 5.1 and 5.4]{Ro10}} or {\cite[Theorem 2.10]{Ni14}}] \label{finite gamma}
Let $\pi$ be an irreducible
cuspidal representation of $\GL_n(\BF_q)$ and $\tau$ an
irreducible generic representation of $\GL_t(\BF_q)$, with $n>t$.
Then there exists a complex number $\gamma(\pi\times \tau,\psi)$
such that
  $$\gamma(\pi\times \tau,\psi)
q^{tk}\sum_{m\in\RU_t\backslash \GL_t(\BF_q)}
 \sum_{x\in \RM_{n-t-k-1,t}}W_\pi
(\left(\begin{smallmatrix}m&0&0\\
x&\RI_{n-t-k-1}&0\\
0&0&\RI_{k+1}\end{smallmatrix}\right)) W_\tau(m)$$
\begin{equation*}\label{gam}
=\sum_{m\in\RU_t\backslash \GL_t(\BF_q)}
\sum_{y\in \RM_{t, k}}  W_\pi
 (\left(\begin{smallmatrix}0&\I_{n-t-k}&0\\
0&0 &\RI_{k}\\
m&0&y \end{smallmatrix} \right))  W_\tau(m)
 ,\end{equation*}
 for all $0\le k\le n-t-1,$
$W_\pi\in \CW(\pi,\psi_n )\text{ and } W_\tau\in
\CW(\tau,\psi_t^{-1}), $
 where $$w_{n,t}= \left(\begin{smallmatrix}
 \RI_{t}& 0\\
 0&w_{n-t} \end{smallmatrix}\right).$$
\end{thm}

\begin{prop}[{\cite[Proposition 4.5]{Ge70}} or
{\cite[Lemma 6.1.1]{Ro10}}]\label{defnB}
Let $\pi$ be an irreducible generic representation of $\GL_n(\BF_q)$
and $\chi_\pi$ its character. Define
$$\CB(g)= {| \RU_n(\BF_q) |}^{-1}\sum_{u\in \RU_n}\psi_n(u^{-1})\chi_{\pi}(gu), \text{ for }g\in \GL_n(\BF_q).$$
Then $ \CB$ satisfies the following  
conditions:
\begin{enumerate}
\item $\CB\in \CW(\pi, \psi_n);$
\item $ \CB(u_1gu_2)=\psi_n(u_1u_2)\CB, \text{ for all }g\in \GL_n(\BF_q), \ u_1,\ u_2\in \RU_n;$
\item $\CB(\RI_n)=1.$
 \end{enumerate}
\end{prop}
We call $\CB$ in Theorem \ref{defnB} the Bessel function of $\pi$ with respect to $\psi$ and denote it by 
$\CB_{\pi, \psi}.$
Then one gains the following nice expression of twisted gamma factors in terms of their Bessel functions.
\begin{prop}[{\cite[Lemma 6.1.4]{Ro10}}]\label{614}
Let $\pi$ be an irreducible cuspidal representation of
$\GL_n(\BF_q)  $ and $\tau$ be an irreducible generic representation of $\GL_r(\BF_q), \ r <n.$ Then \begin{equation*}\label{Bandg}\gamma(\pi\times \tau,\ \psi)=  \sum_{\RU_r\backslash \GL_r(\BF_q)}\CB_{\pi,  \psi}\begin{pmatrix}0&\RI_{n-r}\\
m& 0\end{pmatrix}\CB_{\tau,  \psi^{-1}}(m).
\end{equation*}
\end{prop}


In \cite{Ni17} the first author relates $n\times 1$ local gamma factor with abelian Gauss sums through Green's construction \cite{Gre55}.
Denote by $\widehat\BF_q^\times$ the set of isomorphism classes of characters of $\BF_q^\times$.

\begin{thm}[\cite{Ni17}]\label{gauss sum}
Let $\pi$ be an irreducible cuspidal representation of $\GL_n(\BF_q),$   $n\ge 2 $ and $\tau\in\widehat\BF_q^\times.$ Then
\begin{equation}\label{gauss}
 \gamma(\pi\times \tau,\ \psi)
=(-q^{-1}\tau(-1))^{n-1}\sum_{ a\in \BF^\times_{q^n}}\psi(\Tr a^{-1})\eta_{\pi}(a)\tau(\Nr_{n:1}(a)),\end{equation}
where $\eta_{\pi}$ is the regular character of $\BF_{q^n}^\times$ corresponding to $\pi$ in Green's construction and $\Tr$  denotes the trace map from $\BF_{q^n}\mapsto \BF_q$.

\end{thm}

Note that the above theorem implies that the gamma factors defined by \cite{Ro10} through the Rankin--Selberg convolution is consistent with the ones constructed by Braverman and Kazhdan in \cite[Section 9.2]{BrK00}.

\subsection{Local gamma factors for generic representations of $\GL_n(\CF)$}\label{2.3}

First we review the basic setting of local gamma factors
attached to a pair of irreducible generic representations, for details
of which we refer to~\cite{JP-SS83}.

We denote by~$\Fo_\CF$ the ring of integers in~$\CF$, by~$\Fp_\CF$ the prime
ideal in~$\Fo_\CF$, and by~$\Fk_\CF$ the residue field of~$\CF$,
of cardinality~$q$. We may omit the subscript $_\CF$, when   no confusing about the field $\CF$ may cause.
Denote by $\cInd$ the compact induction functor,
$\RK_n:=\GL_n(\Fo)$ the maximal compact subgroup of $\GL_n(\CF)$, $\FP_n=\RI_n+\RM_n(\Fp)$.
Fix a nontrivial additive character $\psi_\CF$ of $\CF$ such that
$\psi_\CF$ is trivial on $\Fp$ and nontrivial on $\Fo$.
For an element $k\in \Fo,$ define  $\bar k\in \BF_q$ to be its reduction modulo $\Fp.$
Define $$\psi(\bar k)=\psi_\CF(k), \text{ for }k\in \Fo.$$
Then $\psi$ is a nontrivial additive character of $\BF_q.$
Let $\Phi_n$ be the standard non-degenerate character of $\RU_n(\CF)$ defined in Definition
\ref{nonde} in terms of $\psi_\CF.$
Define $$\psi_n(\bar k)=\Phi_n(k), \text{ for }k\in \RK_n\cap \RU_n(\CF).$$
Then $\psi_n$ is the standard non-degenerate character of $\RU_n(\BF_q)$
in terms of $\psi.$

Let~$n,\ t\ge 1$ be integers and let~$\pi$ and~$\tau$ be irreducible
generic representations of~$\GL_n(\CF)$  and~$\GL_t(\CF)$ respectively, with the 
central characters~$\omega_\pi$ and~$\omega_\tau$
respectively. Let~$W_\pi\in \CW(\pi,\Phi_n)$ be a Whittaker function
of~$\pi$ and~$W_\tau\in \CW(\tau,\Phi_t^{-1})$ be a Whittaker function
of~$\tau$. Since it is the only case of interest to us here, we
suppose that~$n>t$.

Let $\RM_{k\times m}$ denote the set of $k\times m$ matrices.
If~$j$ is an integer for which~$n-t-1\ge j\ge 0$, a local zeta integral for the
pair~$(\pi,\tau)$ is defined by
\[
\CZ(W_\pi, W_\tau,   s; j):=
\int_{g}\int_{x} W_\pi
\begin{pmatrix}g&0&0\\
x&\RI_{n-t-1-j}&0\\
0&0&\RI_{j+1}\end{pmatrix}
W_\tau(g)|\det g|^{s-\frac{n-t}{2}} dx dg,
\]
where the integration in the variable~$g$ is over~$\RU_t(\CF)\backslash\GL_t(\CF)$ and the
integration in the variable~$x$ is over $\RM_{(n-t-1-j)\times t}(\CF)$.

For~$g\in\GL_n(\CF)$, we denote by~$R_g$ the right translation action
of~$g$ on
functions from~$\GL_n(\CF)$ to~$\BC$, and we put~$w_{n,t}=
\left(\begin{smallmatrix}
\RI_{t}& 0\\
0&w_{n-t} \end{smallmatrix}\right)$.
In~\cite{JP-SS83}, Jacquet, Piatetski-Shapiro, and Shalika proved  the
following celebrated theorem.
\begin{thm}[{\cite[Section 2.7]{JP-SS83}}]\label{FE}
With notation as above, the following hold.
\begin{enumerate}
\item[(i)] Each integral~$\CZ(W_\pi, W_\tau,   s; j)$ is absolutely convergent
for~$\Re(s)$ sufficiently large and is a rational function of~$q^{-s}$.
\item[(ii)] There exists a rational function in $q^{-s}$
   such that
\begin{equation}\label{gamF}
\CZ(R_{w_{n,t}}\widetilde{W}_\pi, \widetilde{W}_\tau, 1-s; n-t-j-1)\end{equation}
$$=\omega_\tau(-1)^{n-1}\Gamma(s,\pi\times \tau, \psi_\CF)\CZ(W_\pi, W_\tau,  s;j).$$

\end{enumerate}
\end{thm}

\begin{cor}[Corollary 2.7, \cite{JiNiS15}] \label{char} Let~$\pi_1$,~$\pi_2$ be irreducible generic
representations of~$\GL_n(\CF)$. If their local gamma factors~$\Gamma(s,\pi_1\times\chi,\psi_\CF)$
and~$\Gamma(s,\pi_2\times\chi,\psi_\CF)$ are equal as functions in the
complex variable~$s,$ for any character~$\chi$ of~$\CF^\times$,
then $\pi_1$ and $\pi_2$ possess the same central character.

\end{cor}

Recall from~\cite[Section 5]{PS08} the formulation of
generalized Bessel functions for the $p$-adic case.
Let~$\CU\subset\CM\subset\CK$ be compact open subgroups of~$\CK$.
Let~$\tau$ be an irreducible smooth representation of~$\CK$ and let~$\Psi$ be a
linear character of~$\CU$. Take an open normal subgroup~$\CN$ of~$\CK$, which is
contained in~$\Ker(\tau)\cap\CU$.
Let~$\chi_\tau$ be the (trace) character of~$\tau$. The associated
 {\bf generalized Bessel function}~$\CJ:\CK\rightarrow\BC$ of~$\tau$ is defined by
\begin{equation}\label{Bessel}
\CJ(g):=[\CU:\CN]^{-1}\sum_{u\in\CU/\CN}\Psi(u^{-1})\chi_\tau(gu).
\end{equation}
This is independent of the choice of~$\CN$.

\begin{prop}[{\cite[Proposition~5.3]{PS08}}]\label{bs30}
Assume that the data introduced above satisfy the following:
\begin{itemize}
\item $\tau|_\CM$ is an irreducible representation of~$\CM$; and
\item $\tau|_\CM\cong \Ind^\CM_\CU(\Psi)$.
\end{itemize}
Then the generalized Bessel function~$\CJ$ of~$\tau$ enjoys the following properties:
\begin{enumerate}
\item $\CJ(1)=1$;
\item $\CJ(hg)=\CJ(gh)=\Psi(h)\CJ(g)$ for all~$h\in\CU$ and~$g\in\CK$;
\item if~$\CJ(g)\neq 0$, then~$g$ intertwines~$\Psi$; in particular, if~$m\in\CM$, then~$\CJ(m)\neq0$ if and only if~$m\in\CU$;
\end{enumerate}
\end{prop}

 \subsection{Supercuspidal representations of $\GL_n(\CF)$}

 For the rest of this section, we refer to \cite{BuH98} for undefined notation involving construction of supercuspidal representations of $\GL_n(\CF).$
Let~$\pi$ be an irreducible \emph{unitarizable} supercuspidal representation of $\GL_n(\CF)$.
By \cite[Proposition 1.6]{BuH98}, there is an extended maximal simple
type~$(\bJ_\pi,\Lambda_\pi)$ in~$\pi$ such that
\[
\Hom_{\RU_n\cap\bJ_\pi}(\Phi_n,\Lambda_\pi)\neq0.
\]
Since~$\Lambda_\pi$ restricts to a multiple of some simple
character~$\theta_\pi\in\CC(\FA,\beta,\psi_\CF)$, one obtains that~$\theta_\pi(u)=\Phi_n(u)$ for
all~$u\in\RU_n\cap H_\pi^1$. As in~\cite[Definition 4.2]{PS08}, one defines a
character~$\Psi_n:(J_\pi\cap\RU_n)H_\pi^1\rightarrow\BC^\times$ by
\begin{equation}\label{Psin}
\Psi_n(uh):=\Phi_n(u)\theta_\pi(h),
\end{equation}
for all~$u\in J_\pi\cap\RU_n$ and~$h\in H_\pi^1$. By~\cite[Theorem 4.4]{PS08}, the data
\[
\CK_\pi=\bJ_\pi,~\tau_\pi=\Lambda_\pi,~\CM_\pi=(J_\pi\cap\RP_n)J_\pi^1,~\CU_\pi=(J_\pi\cap\RU_n)H_\pi^1,~\text{and}~\Psi=\Psi_n
\]
satisfy the conditions in Proposition~\ref{bs30} and hence define a generalized Bessel function~$\CJ$.

Now we define a function~$\RB_\pi:\GL_n(\CF)\to\BC$ by
\begin{equation}\label{Bwh}
\RB_\pi(g):=\begin{cases}
\Phi_n(u)\CJ_\pi(j)&\text{ if }g=uj\text{ with }u\in\RU_n,~j\in\bJ_\pi,\\
0&\text{ otherwise},
\end{cases}
\end{equation}
which is well-defined by Proposition~\ref{bs30}(ii).
Then, by~\cite[Theorem 5.8]{PS08},~$\RB_\pi$ is a Whittaker function for~$\pi$.
By
Proposition~\ref{bs30}, the restriction of~$\RB_\pi$ to~$\RP_n$ has a particularly
simple description: for~$g\in\RP_n$,
\begin{equation}\label{eqn:WsponPn}
\RB_\pi(g)=\begin{cases}
\Psi_n(g)&\text{ if }g\in(J_\pi\cap\RU_n)H^1_\pi;\\
0&\text{ otherwise}.

\end{cases}
\end{equation}

A level zero supercuspidal representation of $\GL_n(\CF)$,
is given by $$\pi \cong \cInd_{\RZ_n(\CF)\RK_n(\CF)}^{\GL_n(\CF)}\omega\tau, $$ for some   irreducible cuspidal representation $\tau$ of $\GL_n(\BF_q)$, where $\omega$ is a character of $\RZ_n$ satisfying
$\omega|_{\RZ_n\cap \RK_n}(k)=\tau(\bar k)$ and
$q$ is the cardinality of the residue field of $\CF$.

\begin{thm}\label{level0} For $i=1,\ 2,$ let 
$\pi_i  \cong \cInd_{\RZ_{n_i}(\CF)K_{n_i}(\CF)}^{\GL_{n_i}(\CF)}\omega_i\tau_i,$  be
 irreducible, level zero  supercuspidal, unitarizable representations of $\GL_{n_i}(\CF) $,
$n_1>n_2 .$
Then  $$\omega_{2}(-1)^{n_1-1}	\Gamma(s,\pi_1\times \pi_2, \psi_\CF)=q^{\frac{n_2(n_1-n_2-1)}{2}}\gamma(\tau_1\times\tau_2, \psi).$$
\end{thm}
Remark that the depth zero supercuspidal representations are of conductor 0 and then their gamma factors are complex numbers instead of rational functions in $q^s$.
\begin{proof}
In the level zero case, $J_{\pi_i}=\RK_{n_i}, \bJ_{\pi_i}=\RZ_{n_i}\RK_{n_i}, J_{\pi_i}^1=H_{\pi_i}^1=\FP_{n_i}, \theta_{\pi_i}=1,$
and $\CN_{\pi_i}=\FP_{n_i}$.
Since $\CU_{\pi_i}/\CN_{\pi_i}\cong \RU_{n_i}(\BF_q),$ the generalized Bessel function for $\tau_i$ is given by
\begin{eqnarray*}\CJ_i(g)&=&[\CU_{\pi_i}:\CN_{\pi_i}]^{-1}\sum_{u\in\CU_{\pi_i}/\CN_{\pi_i}}\Psi(u^{-1})\chi_{\tau_i}(gu) \\
&=&\frac{1}{|\RU_{n_i}(\BF_q)|}\sum_{u\in\CU_{\pi_i}/\CN_{\pi_i}}\Psi(u^{-1})\chi_{\tau_i}(gu),\text{ for }i=1,\ 2.\end{eqnarray*}
The Whittaker function $\RB_{\pi_i} ,$ defined in Eq. \eqref{Bwh},
are supported in $\RU_{n_i} \RK_{n_i}\RZ_{n_i}.$
Note that  for $ u\in\RU_{n_i} \cap \RK_{n_i},\ k\in \RK_{n_i},\
\CJ_{\pi_i}(ku)=\Phi_{n_i}(u)\CJ_{\pi_i}(k),$  so
\begin{equation}\label{J=B}
\CJ_{\pi_i}(k)=\CB_{\tau_i}(\bar k), \text{ for }k\in \RK_{n_i},\end{equation}
where $\CB_{\tau_i }$ is the  Bessel function of $\tau_i$ in terms of $\psi_{n_i}.$   

Since the support of $\RB_{\pi_i}\subset \RU_{n_i}(\CF)\RK_{n_i}\RZ_{n_i}(\CF),$ by Eq. \eqref{Bwh} and \eqref{J=B},\small
  \begin{eqnarray*}&&
 \int_{g\in {\RU_{n_2}(\CF)\cap \RK_{n_2}}\backslash\RK_{n_2}}\int_{x} \RB_{\pi_1}
\begin{pmatrix}g&0&0\\
x&\RI_{n_1-n_2- j-1}&0\\
0&0&\RI_{j+1}\end{pmatrix}
\RB_{\pi_2} (g)|\det g|^{s-\frac{n_1-1}{2}} dx dg\\
&=&\Vol(\RU_{n_2}(\Fo)\FP_{n_2})\Vol(\RM_{(n_1-n_2-j-1)\times n_2}(\Fp))\\
&&\times
\sum_{\bar{g}\in \RU_{n_2}(\BF_q)\backslash\GL_{n_2}(\BF_q),\ \bar{x}\in \RM_{(n_1-n_2-j-1)\times n_2}(\BF_q)}\CB_{\tau_1}
\begin{pmatrix}\bar{g}&0&0\\
\bar{x}&\RI_{n_1-n_2-j-1}&0\\
0&0&\RI_{j+1}\end{pmatrix}
\CB_{\tau_2}(\bar{g}),
\end{eqnarray*}
\normalsize
where $\Vol(S)$ denotes the volume of the set $S$ relative to the corresponding Haar measure and the integral domain of $x$ is $\RM_{ (n_1-n_2-j-1)\times n_2}(\Fo)$.
Similarly,
  \begin{eqnarray*}&&
 \int_{g}\int_{x} \RB_{\pi_1}
\begin{pmatrix}0&\RI_{n_1-n_2-j}&0\\
0&0&\RI_{j}\\
g&0&x\end{pmatrix}
\RB_{\pi_2}(g)|\det g|^{s-\frac{n_1-1}{2}} dx dg\\
&=& \Vol(\RU_{n_2}(\Fo)\FP_{n_2}) \Vol(\RM_{n_2\times j}(\Fp))\\
&&\times
\sum_{\bar{g}\in \RU_{n_2}(\BF_q)\backslash\GL_{n_2}(\BF_q),\ \bar{x}\in \RM_{n_2\times j}(\BF_q)}\CB_{\tau_1}
\begin{pmatrix}0&\RI_{n_1-n_2-j}&0\\
0&0&\RI_{j}\\
\bar{g}&0&\bar{x}\end{pmatrix}
\CB_{\tau_2}(\bar{g}),
\end{eqnarray*}
where the integral domains of $g$ and $x$ are $\RU_{n_2}(\CF)\cap \RK_{n_2}\backslash\RK_{n_2}$ and $\RM_{n_2\times j}(\Fo)$ respectively.

By Proposition 23.1 \cite{BuH06}, the self dual Haar measure on
$\CF$ is given by $\Vol(\Fo)=q^{\frac{1}{2}}.$
Then by the functional equation in Theorems \ref{FE} and \ref{finite gamma},
$$\omega_2(-1)^{n_1-1}	\Gamma(s,\pi_1\times \pi_2, \psi_\CF)=q^{\frac{n_2(n_1-n_2-1)}{2}}\gamma(\tau_1\times\tau_2, \psi).$$
\end{proof}

\section{Main results}\label{sec:main}
This section is devoted to the proof of the main results.
Since Gauss sums has been well developed for $\widehat{\BF_{p^n}^\times}$ instead of   $\widehat{\BF_{q^n}^\times},$ for $q=p^r$ with $r>1.$ Our main results are restricted to the prime field case $\BF_p$.

Let $p$ be a prime. 
Choose the nontrivial additive character $\psi$ of $\BF_p$ such that $\psi(1)=\varepsilon,$ where $\varepsilon$ is a primitive $p$-th root of unity. 
For a character $\chi$ of $\BF^\times_{p^n}$, define the Gauss sum $S(\chi)$ by
$$
S(\chi)=\sum_{x\in \BF^\times_{p^n}}\chi(x)\psi(\Tr(x))\in \BQ(\mu_{p^n-1},\mu_{p}),
$$ 
where $\Tr$ is the trace map from $\BF_{p^n}$ to $\BF_{p}$ and $\mu_N$ is the group of all $N$-th roots of unity.

In terms of Green's construction, Theorem \ref{gauss sum} and Lemma \ref{lm:central-character}, we recall Conjecture \ref{conj:Gauss-sum}.
\begin{conj}[Local Converse Theorem on Gauss Sums]\label{conj2}
Let $\chi_1$ and $\chi_2$ be two regular characters of $\BF_{p^n}^\times$ and $n$ be a prime. 
If
$$
S(\chi_1\cdot \sig)=S(\chi_2\cdot \sig),
\text{ for all }\sig\in \wh{\BF_p^\times},$$
then $\chi_1=\chi^{p^i}_2$ for some integer $i$.
\end{conj}

In this section, we will verify Conjecture \ref{conj2} for $n\leq 5$.

\subsection{Gross-Koblitz formula}\label{GK}
For the readers' convenience,  we recall Stickelberger's Theorem and Gross-Koblitz formula in this section.
One may refer to \cite{La90} for instance.
In this section, we will prepare our technical tools on the way to prove Theorem \ref{thm:main-5}.

Define the Teichm\"uller character
$$
\omega:\BF_{p^n}^\times\mapsto  \mu_{p^n-1}\text{  by }\omega(u)\equiv u \bmod{\Fp},
$$
where $\Fp$ is a prime ideal in $\BQ(\mu_{p^n-1})$ lying above the prime number $p$.
In fact, the residue class field of modulo
$\Fp$ is identified with $\BF^\times_{p^n}$, 
and there exists an isomorphism between $\mu_{p^n-1}$ and $\BF^\times_{p^n}$.
Moreover, $\omega$ generates the character group of $\BF^\times_{p^n}$.

For an integer $k$, 
define its {\bf $p$-adic expansion  (module $p^n-1$)} by
$$
k\equiv k_0+k_1p+\cdots+k_{n-1}p^{n-1}\bmod{p^n-1}
$$
where $0\leq k_i\leq p-1$. Define
$$
s(k)=k_0+k_1+\cdots+k_{n-1}
\text{ and }
t(k)=k_0!k_1!\cdots k_{n-1}!.$$
Then $s(\cdot)$ and $t(\cdot)$ are $(p^n-1)$-periodical.

\begin{thm}[{Stickelberger's Theorem, \cite[Theorem 1.2.1]{La90}}]\label{Stickel}
For any integer $k$, we have the following congruence:
$$
S(\omega^{-k})\equiv -\frac{(\varepsilon-1)^{s(k)}}{t(k)}\bmod{\FP},
$$
where $\FP$ is a prime ideal lying above $\Fp$ in $\BQ(\mu_{p^n-1},\mu_p)$.
Moreover, $$\ord_\FP S(\omega^{-k})=s(k).$$
\end{thm}

Recall that for a $p$-adic integer $z$ in $\BZ_p$ (the ring of integers of  $\BQ_p$), the $p$-adic gamma function is defined by 
$$
\Gamma_p(z)=\lim_{m\to z}(-1)^m\prod_{ 0<j<m,~ (p,j)=1}j.
$$

\begin{thm}[{Gross-Koblitz formula, \cite{GrKo79} or  \cite[Theorem 15.4.3]{La90}}]\label{GKformula}
One has
$$
S(\omega^a)=(-1)^n p^n \pi^{-s(a)}\prod^{n-1}_{i=0}\Gamma_p(1-\langle\frac{p^ia}{p^n-1}\rangle),
$$ 
where $\pi$ is an element in the  $p$-adic complex field $\BC_p$ such that $\pi^{p-1}=-p$
and 
 $\apair{t}$ is the smallest non-negative real number in the residue class modulo $  {\BZ} $ for $t\in \BR$.
\end{thm}

\begin{lem}\label{lm:central-character}
If $S(\omega^{-\alpha})=S(\omega^{-\beta})$,
then $\omega^{-\alpha}$ and $\omega^{-\beta}$ admit the same restriction to $\BF^\times_p$.
\end{lem}
\begin{proof}
Let $g$ be a generator of the multiplicative group $\BF^\times_{p^n}$.
Then $g^{\frac{p^n-1}{p-1}}$ is a generator of $\BF^\times_p$.
It is enough to show that $\omega^{-\alpha}(g^{\frac{p^n-1}{p-1}})=\omega^{-\beta}(g^{\frac{p^n-1}{p-1}})$.

Note 
\begin{align*}
\alpha\equiv\beta\bmod{p-1}\iff
 & \alpha\frac{p^n-1}{p-1}\equiv \beta\frac{p^n-1}{p-1}\bmod{p^n-1}\\
\iff&\omega^{-\alpha}(g^{\frac{p^n-1}{p-1}})=\omega^{-\beta}(g^{\frac{p^n-1}{p-1}}).
\end{align*}
Let $\alpha\equiv(\alpha_1,\dots,\alpha_n)$ and $\beta\equiv (\beta_1,\dots,\beta_n)$ be the $p$-adic expansions. 
Then $\alpha\equiv \sum^{n}_{i=1}\alpha_i=s(\alpha)\bmod{p-1}$ and $\beta\equiv s(\beta)\bmod{p-1}$.
Since $S(\omega^{-\alpha})=S(\omega^{-\beta})$,  $s(\alpha)=s(\beta)$ by Theorem \ref{Stickel} and we have $\alpha\equiv\beta\bmod{p-1}$. It completes the proof of this Lemma.
\end{proof}

For a character $\chi$ of $\BF_p^\times$, we have 
$\omega^{-\alpha}\otimes(\chi\circ\Nr_{n:1})=\omega^{-(\alpha+\hat{k})}$
for some $0\leq k< p-1$, where $\hat{k}=k\cdot \frac{p^n-1}{p-1}$.
Now we rewrite Conjecture \ref{conj2} as follows.
\begin{conj}\label{conj:main}
Suppose that $\omega^{-\alpha}$ and $\omega^{-\beta}$ are regular characters of $\BF_{p^n}^\times$ and $n$ is a prime.
If
\begin{equation}\label{eq:S-twist}
S(\omega^{-(\alpha+\hat k)})=S(\omega^{-(\beta+\hat k)})
\end{equation}
for all $0\leq k< p-1$, then
$\alpha\equiv p^{j}\beta\bmod{p^n-1}$ for some $j$.
\end{conj}

For positive integers $n$ and $\alpha$,
denote the $p$-adic expansion of $\alpha$ by
$$
\alpha\equiv (\alpha_1,\dots,\alpha_n),
$$
when $ \alpha\equiv \sum_{i=1}^{n}\alpha_i p^{i-1}\bmod{p^n-1},$ where $0\le \alpha_i\le p-1$ for all $i$.
For any integer $j$, define $\alpha_j=\alpha_i$ for  the unique $1\leq i\leq n$ such that $j\equiv i\bmod{n}$. 
Define
$$\|\alpha\|:=
(\|\alpha\|_1,\|\alpha\|_2,\|\alpha\|_3,\dots, \|\alpha\|_{r(\alpha)} )
$$
where $\|\alpha\|_i$ is the $i$-th largest  element and $r(\alpha)$ is the number of distinct elements in   $\{\alpha_1,\dots,\alpha_n\}$ respectively.
Denote $$m_i(\alpha)=\#\{1\le j\le n\ |\ \alpha_j=\|\alpha\|_i\} $$ and
$m(\alpha)=(m_1(\alpha),\cdots, m_{r(\alpha)}(\alpha)).$
 
\begin{defn}
We call  $\omega^{-\alpha}$ (resp.~$\alpha$) and  $ \omega^{-\beta} $ (resp.~$\beta$)  {\bf distinguishable}  if $S(\omega^{-(\alpha+\hat k)})\ne S(\omega^{-(\beta+\hat k)})$ for some $0\leq k< p-1.$ 
We say that $\omega^{-\alpha}$ is {\bf distinguishable} if $\omega^{-\alpha}$ is distinguishable from any regular character $\omega^{-\beta} $. 
In this case, we also call $\alpha$ distinguishable. 
 
We call a collection of sets $S_i,$ consisting of non-equivalent regular characters of $\BF^\times_{p^n}$, {\bf distinguishable sets}, if $\alpha$ and $\beta$ are distinguishable whenever $\alpha\in S_j$ and $\beta\in S_k$ for any $j\ne k.$ 
\end{defn}

\begin{rmk}\label{rmk:distinguishable}
By Green's construction, $\omega^{-\alpha}$ and $\omega^{-\beta}$ are equivalent, if and only if they corresponds to the same cuspidal representation of $\GL_n(\BF_q)$. 
Note that $S(\chi^{p^j})=S(\chi)$ for any character $\chi$ of $\BF^{\times}_{p^n}$  
and the $p$-adic expansion of $p^j\alpha$
\begin{equation}\label{eq:shift}
p^j\alpha\equiv (\alpha_{1+j},\alpha_{2+j},\dots,\alpha_{n+j}) \bmod{p^n-1}.	
\end{equation}
Now, to verify Conjecture \ref{conj:main} it suffices to show that all non-equivalent regular characters are distinguishable. 
\end{rmk}

In the following lemmas, we assume that    $\omega^{-\alpha}$ and $\omega^{-\beta}$ are regular characters of $\BF^\times_{p^n}$.   Let
$$\alpha\equiv (\alpha_1,\cdots,\alpha_{n}) \text{ and }  \beta\equiv (\beta_1,\cdots,\beta_{n}).$$

\begin{lm} \label{sandt}\label{symm}
Suppose that  $\omega^{-\alpha}$ and $\omega^{-\beta}$ are regular.
If $S(\omega^\alpha)=S(\omega^\beta)$,
then we have
\begin{enumerate}
	\item  $s(\alpha)=s(\beta)$ and $t(\alpha)\equiv t(\beta)\bmod{p}$;
	\item $S(\omega^{-\alpha})=S(\omega^{-\beta})$.
\end{enumerate}
\end{lm}
\begin{proof}
Part (1) follows from Stickelberger's Theorem in Theorem \ref{Stickel}.

Referring to \cite[{\bf GS2} Section 1.1]{La90}, one has
$S(\omega^{-\alpha})= \omega^\alpha(-1)\overline{S(\omega^\alpha)}$.
By Lemma \ref{lm:central-character},  $\omega^{-\alpha}$ and $\omega^{-\beta}$ admit the same restriction to $\BF^\times_{p}.$
It follows that  $S(\omega^{-\alpha})=S(\omega^{-\beta})$ if and only if $S(\omega^{\alpha})=S(\omega^{\beta})$.
\end{proof}

To describe the $p$-adic expansion of $\alpha+\widehat{p-a}$, we introduce the following directed graph of the $p$-adic expansion modulo $p^n-1$.
\begin{defn}[The $a$-th directed Graph of $p$-adic expansions]
Let $\alpha\equiv(\alpha_1,\alpha_2,\dots,\alpha_n)$ be the $p$-adic expansion modulo $p^n-1$.
For any $1\leq a\leq p-1$, 
we define the $a$-th directed graph $\CG_a(\alpha)$ of $\alpha$ as follows: for  $1\le k\le n$, each $\alpha_k$ is a vertex; if $\alpha_k<a-1$, then there is no edge connected to $\alpha_k$;
if $\alpha_k\geq a$ and $\alpha_{k+1}\geq a-1$, then we assign a directed edge from $\alpha_k$ to $\alpha_{k+1}$, denoted by $\alpha_k\stackrel{a}{\rightarrow} \alpha_{k+1}$;
if $\alpha_k=\alpha_{k+1}=a-1$, then we assign $\alpha_k\stackrel{a}{\rightarrow} \alpha_{k+1}$ only when  $\alpha_{k-1}\stackrel{a}{\rightarrow} \alpha_k$.
\end{defn}

Let $\CG_{a}(\alpha)^\circ$ denote the subgraph of $\CG_{a}(\alpha)$, which consists of connected components   containing at least one vertex  $\geq a$.
Denote by $v_a(\alpha)$  the numbers of vertices  in $\CG_{a}(\alpha)^\circ$.

\begin{lem}
For $\alpha\equiv (\alpha_1,\alpha_2,\dots,\alpha_n)$ and any $1\leq a\leq p-1$, we have
\begin{equation}\label{eq:alpha+p-a}
s(\alpha+\widehat{p-a})=s(\alpha)+(p-a)n-v_a(\alpha)(p-1).	
\end{equation}
\end{lem}
\begin{proof}
If $\CG_a(\alpha)^\circ=\{\alpha_i\ | \ i\in I\}$ is connected, then $\alpha_i\geq a-1$ and $(\alpha+\widehat{p-a})_i=\alpha_i-a+1$ for all $i\in I$.
It follows that $s(\alpha+\widehat{p-a})= s(\alpha)-n(a-1)$ and Eq. \eqref{eq:alpha+p-a} holds in this case.

Assume that $\CG_a(\alpha)^\circ$ is not connected.
For each connected component in $\CG_{a}(\alpha)^\circ$
\begin{equation}\label{eq:CG-connected-comp}
\alpha_i\stackrel{a}{\rightarrow} \alpha_{i+1}\stackrel{a}{\rightarrow}\cdots \stackrel{a}{\rightarrow} \alpha_{i+j-1};
\end{equation}
 $\alpha_i\geq a$; $\alpha_{i+j}<a-1$, and either $\alpha_{i-1}<a-1$ or $\alpha_{i-j}=a-1,$ for $j=1,\cdots, \ i_0,$ and $\alpha_{i-i_0-1}\le a-2$ for some $i_0\ge 1$.
It follows that if $\alpha_k\notin\CG_{a}(\alpha)^\circ$, $(\alpha+\widehat{p-a})_k=\alpha_k+p-a$ or $\alpha_k+p-a+1$. The latter case happens when $\alpha_k=\alpha_{i+j}$ is next to some connected component \eqref{eq:CG-connected-comp} on the right end.
If $\alpha_k\in\CG_{a}(\alpha)^\circ$, then $(\alpha+\widehat{p-a})_k=\alpha_k-a+1\geq 0$  unless $\alpha_k=\alpha_i$ is the leftmost vertex of the connected component \eqref{eq:CG-connected-comp},
in which case $(\alpha+\widehat{p-a})_i=\alpha-a\geq 0$.
Overall, we have 
\begin{align*}
s(\alpha+\widehat{p-a})=&s(\alpha)+(p-a)(n-v_a(\alpha))-(a-1)v_a(\alpha)\\
=&s(\alpha)+(p-a)n-(p-1)v_a(\alpha),
\end{align*}
which completes the proof this lemma.
\end{proof}

\begin{lem}\label{r1} \label{ru}
Let  $\omega^{-\alpha}$ and $\omega^{-\beta}$ be regular.
If $S(\omega^{-(\alpha+\hat{k})})=S(\omega^{-(\beta+\hat{k})})$ for all $0\leq k< p-1$, then $\|\alpha\|_1=\|\beta\|_1$  and $\|\beta\|_{r(\beta)}=\|\alpha\|_{r(\alpha)}$.
\end{lem}
\begin{proof}
Since $S(\omega^{-(\alpha+\hat{k})})=S(\omega^{-\beta+\hat{k}})$ for $0\leq k< p-1$, we have
$s(\alpha+\hat{k})=s(\beta+\hat{k})$ for all $1
\leq k<p-1$.
Assume on the contrast that $\|\alpha\|_1\ne  \|\beta\|_1$.
By symmetry, it suffices to consider the case $\|\alpha\|_1>  \|\beta\|_1$. 
Take $k=p-\|\alpha\|_1$ and then $1\leq k<p-1$. 
As $\|\alpha\|_1> \|\beta\|_1$, $0<\beta_i+p-\|\alpha\|_1\leq p-1$
and 
$$
\beta+\hat{k}\equiv (\beta_1+p-\|\alpha\|_1,\beta_2+p-\|\alpha\|_1,\dots, \beta_n+p-\|\alpha\|_1).
$$
Hence $s(\beta)+nk=s(\beta+\hat k)$.
By Eq. \eqref{eq:alpha+p-a}, one has 
$$s(\alpha+\hat k)= s(\alpha)+kn-v_{\|\alpha\|_1}(\alpha)(p-1).$$
Since $\{\alpha_j\mid \alpha_j=\|\alpha\|_1\}\subset \CG_{\|\alpha\|_1}(\alpha)^\circ$,   $v_{\|\alpha\|_1}(\alpha)\geq m_1(\alpha)>0$.
Then 
$$s(\alpha+\hat k)\le s(\alpha)+kn-v_{\|\alpha\|_1}(\alpha)(p-1)<s(\beta)+kn=s(\beta+\hat k).$$ 
It contradicts to the fact that $s(\alpha+\hat k)=s(\beta +\hat k) $ and $s(\alpha)=s(\beta)$ following Lemma~\ref{sandt}.
Thus $\|\alpha\|_1=\|\beta\|_1$.

Next, by Lemma \ref{symm}, we have $S(\omega^{-(-\alpha+\hat k)})=S(\omega^{-(-\beta+\hat k)})$ for all $0\le k< p-1$. 
As $-\alpha\equiv (p-1-\alpha_1,\dots, p-1-\alpha_n)$ and $-\beta\equiv (p-1-\beta_1,\dots, p-1-\beta_n)$, we have $\|\alpha\|_{r(\alpha)}=\|\beta\|_{r(\beta)}$ due to $\|-\alpha\|_1=\|-\beta\|_1$, which completes the proof.
 \end{proof}

\begin{lem}\label{lm:multi-set-6}
Let  $\omega^{-\alpha}$ and $\omega^{-\beta}$ be regular.
Suppose that $S(\omega^{-(\alpha+\hat{k})})=S(\omega^{-(\beta+\hat{k})})$ for all $0\leq k< p-1$.
Then $\{(\alpha+\hat{k})_i\ |\  1\le i\le n\} =\{(\beta+\hat{k})_i\ |\  1\le i\le n\}  $ as multi-sets for all $0\leq k<p-1$, when $n\le 5$. 
\end{lem}
\begin{proof}
It suffices to show that $\{\alpha_i\mid 1\leq i\leq n\}=\{\beta_i\mid 1\leq i\leq n\}$ as multi-sets.
Write 
$$\{a_1\ge \cdots\ge a_n\}=\{\alpha_i \mid 1\le i\le n\} 
\text{ and }
\{b_1\ge \cdots\ge b_n\}=\{\beta_i \mid  1\le i\le n\}$$ as multi-sets.
First, we prove  
\begin{equation}\label{eq:multi-1}
|a_i-b_i|\le 1 \text{ for all }1\le i\le n.	
\end{equation}
Assume on the contrast $a_m-b_m\ge 2.$
Since $\{a_1,\dots,a_m\}$ are vertices in $\CG_{a_m}(\alpha)^\circ$, we have $m\leq v_{a_m}(\alpha)$.
By $b_m<a_m-1$, $\{b_m,b_{m+1},\dots,b_n\}$ are not in $\CG_{a_m}(\beta)^\circ$ and then  $v_{a_m}(\beta)\leq m-1$.
By Eq. \eqref{eq:alpha+p-a}, we have 
\begin{align*}
&s(\alpha+\widehat{p-a_m})\le s(\alpha)+n(p-a_m)-(p-1)m\\
<&s(\alpha)+n(p-a_m)-(p-1)(m-1)\le s(\beta+\widehat{p-a_m}),
\end{align*}
which contradicts to our assumption.

Next, denote by $$d_+:=\#\{i\ |\  a_i>b_i\} 
 \text{ and }d_-:=\#\{i\ |\  a_i<b_i\}.$$
By Eq. \eqref{eq:multi-1} and $s(\alpha)=s(\beta),$
we have $d_+=d_-$. 
Now, it suffices to show that $d_+=0$.
By $n\leq 5$ and Lemma \ref{ru}, one has $a_1=b_1$ and $a_n=b_n$, so $d_+\leq 1$.

If $d_+=1$, we may assume that $a_i=b_i+1$ and $a_j=b_j-1$ for some $i\ne j$. By Lemma \ref{symm}, 
$t(\alpha)/t(\beta)=(b_i+1)/b_j\equiv 1\bmod{p}$.
By  $b_i+1,b_j\leq p-1$, we have $a_i=b_i+1=b_j=a_j+1$, which contradicts with $(a_i-a_j)(b_i-b_j)\geq 0$.

\end{proof}

We call $\alpha$ {\bf max-consecutive} if there exists an integer $n_{1,\alpha}$  such that $1\leq n_{1,\alpha}\leq n$ and $\{n_{1,\alpha},n_{1,\alpha}+1,\dots,n_{1,\alpha}+m_1(\alpha)-1\}\equiv \{j\mid \alpha_j=\|\alpha\|_1\} \bmod{n}$.
We call $\alpha$ {\bf mini-consecutive} if there exists an integer $n_{r,\alpha}$  such that $1\leq n_{r,\alpha}\leq n$ and $\{n_{r,\alpha},n_{r,\alpha}+1,\dots,n_{r,\alpha}+m_r(\alpha)-1\}\equiv \{j\mid \alpha_j=\|\alpha\|_r\} \bmod{n}$, where $r=r(\alpha)$.
It is easy to check that such $n_{1,\alpha}$ (resp. $n_{r,\alpha}$) is unique.

\begin{lem}\label{lm:consecutive}
Let  $\omega^{-\alpha}$ and $\omega^{-\beta}$ be regular and $\omega^{-\alpha}\vert_{\BF^\times_p}=\omega^{-\beta}\vert_{\BF^\times_p}$.
Suppose that the entries of $\alpha+\hat k$ and $\beta+\hat k$ are equal as multi-sets for all $0\le k< p-1.$
Assume that $\|\alpha\|_1-\|\alpha\|_r>1, $ where $r=r(\alpha)$.
If 
 $\alpha$ is max-consecutive, 
then $\beta$ is max-consecutive, and $m_1(\alpha)=m_1(\beta);$
  $\alpha_{n_{1,\alpha}+m_1(\alpha)}=\beta_{n_{1,\beta}+m_1(\beta)}$.
By symmetry, if  
 $\alpha$ is mini-consecutive, then $\beta$ is mini-consecutive, and $m_r(\alpha)=m_r(\beta);$
 $\alpha_{n_{r,\alpha}+m_r(\alpha)}=\beta_{n_{r,\beta}+m_r(\beta)}.$  
 \end{lem}

\begin{proof}
 
Without loss of generality,  assume that $\{j\mid \alpha_j=\|\alpha\|_1\}=\{1,2,\dots,m_1(\alpha)\}$ and $\alpha_n\ne \alpha_1=\beta_1=\|\alpha\|_1\ne \beta_n$. Otherwise, we may replace $\alpha$ and $\beta$ by $p^j\alpha$ and $p^{j'}\beta$ for some $j$ and $j'$, since
 the assumption also holds for  $p^j\alpha$ and $p^{j'}\beta$ by Eq. \eqref{eq:shift}.

Now, the entries of $\alpha+\hat{k}$ and $\beta+\hat{k}$ are equal as multi-sets.
Take $z=p-a$ where $a=\|\alpha\|_1$ and then $1\leq z<p-1$. 
We consider the $p$-adic expansion of $\alpha+\hat{z}$. 
Since $\CG_{a}(\alpha)^\circ$ is connected, $(\alpha+\hat{z})_i=1$ for all $2\leq i\leq v_{a}(\alpha)$ and
$(\alpha+\hat{z})_i\ne 1$ for other $i$.
As $\{(\alpha+\hat{z})_i\mid 1\leq i\leq n\}=\{(\beta+\hat{z})_i\mid 1\leq i\leq n\}$ as multi-set,  we have $\beta$ is max-consecutive and $m_1(\beta)=m_1(\alpha)$.

If $v_a(\alpha)>m_1(\alpha)$, then $\alpha_{m_1(\alpha)+1}=a-1$ and there is a zero entry other than $\alpha_1'$ in $\alpha+\hat{k}=(\alpha_1',\cdots,\alpha_n').$
To gain a zero entry other than $\beta_1'$ in $\beta+\hat{z}=(\beta_1',\cdots,\beta_n'),$ we see that $\beta_{m_1(\alpha)+1}=a-1.$
If $v_a(\alpha)=m_1(\alpha)$, then $\max\{\alpha_{m_1(\alpha)+1},\beta_{m_1(\alpha)+1}\}<a-1$.
For $k=p-a$, we have the following $p$-adic expansions of $\alpha+\hat{k}$ and $\beta+\hat{z}$:
 $(\alpha+\hat{z})_1=(\beta+\hat{z})_1=0$; $(\alpha+\hat{z})_i=(\beta+\hat{z})_i=1$ if $2\leq i\leq m_1(\alpha)$;  
 $(\alpha+\hat{z})_{m_1(\alpha)+1}=p-a+1+\alpha_{m_1(\alpha)+1}$ and
 $ (\beta+\hat{z})_{m_1(\alpha)+1}=p-a+1+\beta_{m_1(\alpha)+1}$;
 $(\alpha+\hat{z})_i=p-a+\alpha_i$ and  $(\beta+\hat{z})_i=p-a+\beta_i$ for $m_1(\alpha)+1<i\leq n$.
Since $\{\alpha_i\mid 1\leq i\le n \}=\{\beta_i\mid 1\leq i\le n \}$ and
$\{(\alpha+\hat{z})_i\mid 1\leq i\leq n\}=\{(\beta+\hat{z})_i\mid 1\leq i\leq n\}$ as multi-sets,   $\alpha_{m_1(\alpha)+1}=\beta_{m_1(\alpha)+1}$ follows.
\end{proof}

Define
\begin{equation}\label{eq:V-m}
V_m(\alpha)=\prod_{i=1}^{n }(\alpha_i+\alpha_{i+1}p+\cdots+\alpha_{i+m}p^{m})!', 	
\end{equation}
where $N!'=\prod_{i=1, (i,p)=1}^Ni$.
Since we only focus on the  residue classes of $V_m(\alpha)$ modulo $p^{m+1}$, we will abuse the notation and write
$$V_m(\alpha)=V_m(\beta)\text{ if }V_m(\alpha)\equiv V_m(\beta) \bmod{p^{m+1}}.$$

\begin{lem}\label{Usame}
If
$S(\omega^{ \alpha })=S(\omega^{ \beta }),$
then $V_m(\alpha)=V_m(\beta),$ for $m\ge 0.$
\end{lem}
\begin{proof}
By Lemma \ref{symm}, one has $S(\omega^{-\alpha })=S(\omega^{-\beta })$.
Following from Theorem \ref{Stickel}, we have $s(\alpha)=s(\beta)$.
By Theorem \ref{GKformula} and $S(\omega^{ \alpha })=S(\omega^{ \beta })$, 
\begin{equation}\label{eq:Usame-1}
\prod^{n}_{i=1}\Gamma_p(1-\langle\frac{p^{i-1}\alpha}{p^n-1}\rangle)=\prod^{n}_{i=1}\Gamma_p(1-\langle\frac{p^{i-1}\beta}{p^n-1}\rangle).
\end{equation}
Recall that for all positive integers $x$, $y$ and $m$,
$\Gamma_p(x)\equiv \Gamma_p(y) \bmod{p^m}$ if $x\equiv y \bmod{p^m}$ in \cite[Lemma 14.1.1]{La90}.
Then for $1\leq i\leq n$ and a positive integer $a$,
\begin{equation}\label{eq:Usame-2}
\Gamma_p(1-\langle \frac{p^{i-1}a}{q-1}\rangle)\equiv \Gamma_p(1+a_i+a_{i+1}p+\cdots+a_{i+m}p^{m})\bmod{p^{m+1}}.
\end{equation}
Since $\Gamma_p(x)=(-1)^x\prod^{x-1}_{j=1,(p,j)=1}j$ for any integer $x\geq 2$ and $\Gamma_p(1)=-1$, we continue~Eq. \eqref{eq:Usame-2} and obtain
\begin{equation}\label{eq:Usame-3}
\Gamma_p(1-\langle \frac{p^{i-1}a}{q-1}\rangle)\equiv (-1)^{1+a_i+a_{i+1}p+\cdots+a_{i+m}p^{m}}(a_i+a_{i+1}p+\cdots+a_{i+m}p^{m})!'.	
\end{equation}
By Eq. \eqref{eq:Usame-1}, \eqref{eq:Usame-2}, \eqref{eq:Usame-3},
\begin{align*}
 & (-1)^{\sum^n_{i=1}1+\alpha_i+\alpha_{i+1}p+\cdots+\alpha_{i+m}p^{m}}V_m(\alpha)
= (-1)^{n+s(\alpha)(1+p+\cdots+p^m)}V_m(\alpha) \\
=&(-1)^{n+s(\beta)(1+p+\cdots+p^m)}V_m(\beta).
\end{align*}
Then $V_m(\alpha)\equiv V_m(\beta)\bmod{p^{m+1}}$, since $s(\alpha)=s(\beta)$.
It completes the proof.
\end{proof}

\subsection{Case $n=4$} 
In this section, we show that Conjecture \ref{conj:main}
holds for $n=4$, although $n$ is not a prime.
 
\begin{prop}\label{n=4}
Suppose $\omega^{-\alpha}$ and $\omega^{-\beta}$ be regular characters of $\BF_{p^4}^\times.$ 
If $S(\omega^{-(\alpha+\hat k)})=S(\omega^{-(\beta+\hat k)})$ for all  $0\leq k< p-1$, 
then
$\alpha$ and $\beta$ are equivalent.
\end{prop}
\begin{proof}
Since $\omega^{-\alpha}$ and $\omega^{-\beta}$ are regular,
we have $r(\alpha)$ and $r(\beta)> 1$.
By Lemma \ref{lm:multi-set-6}, we have $\|\alpha\|_1=\|\beta\|_1$ and $\|\alpha\|_{r(\alpha)}=\|\beta\|_{r(\beta)}$. We may assume that $\|\alpha\|_{r(\alpha)}=\|\beta\|_{r(\beta)}=0,$
otherwise replace $\alpha$ and $\beta$ by $\alpha-\widehat{\|\alpha\|_{r(\alpha)}}$ and $\beta-\widehat{\|\alpha\|_{r(\alpha)}}$ respectively.
Following from Remark \ref{rmk:distinguishable}, 
for some $j$ and $j'$ we have  
 $$p^j\alpha\equiv (a, \alpha_2, \alpha_3, \alpha_4)\text{ and }p^{j'}\beta\equiv (a, \beta_2, \beta_3, \beta_4)$$
  where $a= \|\alpha\|_1>\alpha_4$ and $a>\beta_4$.
We may also assume that $m_1\ge m_{r(\alpha)}$,
otherwise, replace $\alpha$ and $\beta$ by $-\alpha$ and $-\beta$ respectively.

Assume $a=1$. 
Then $r(\alpha)=r(\beta)=2$ and $\alpha_4=\beta_4=0$. By $s(\alpha)=s(\beta)$, we have $m(\alpha)=m(\beta)$ in $\{(3,1),(2,2)\}$.
If  $m(\alpha)=m(\beta)=(3,1)$, then $\alpha$ and $\beta$ are equivalent.
Suppose $m(\alpha)=m(\beta)=(2,2)$.
Then $\alpha,\beta\in \{(1,0,1,0),(1,1,0,0)\}$, that is, $\alpha,\beta\in \{1+p,1+p^2\}\bmod{p^4-1}$.
Note that $V_1(1+p)\equiv p!'(p+1)!'\bmod{p^2};$ $V_1(1+p^2)\equiv p!'p!'\bmod{p^2}$ and then  $V_1(1+p)\ne V_1(1+p^2)$.
Therefore, $\alpha\equiv \beta\bmod{p^4-1}$ by Lemma \ref{Usame}.

In the rest of this proof, we assume that $a>1$, which implies $p\ne 2$.
We verify the statement in the following 3 cases of $m_1=\#\{\alpha_j\mid \alpha_j=\|\alpha\|_1\}\in\{1,2,3\}$ under the assumption $m_1\geq m_{r(\alpha)}.$

If $m_1=3$, then $\alpha$ and $\beta$ are equivalent.

If $m_1=2$, we only need to consider $(a,a,b,0)$, $(a,b,a,0)$ and $(a,a,0,b)$ for $a>b\ge 0$.
First, $(a,a,b,0)$ and $(a,a,0,b)$ are distinguishable from $(a,b,a,0)$ as the first two are max-consecutive. Next, for $b\ne 0,$ $(a,a,b,0)$ is distinguishable from $(a,a,0,b)$ by Lemma \ref{lm:consecutive} .  


Suppose that $m_1=1$. 
Then $m_{r}=1$ and $\alpha,$  $ \beta$ are both max-consecutive and mini-consecutive. 
 It follows that $\alpha_2=\beta_2$  and $\alpha_{j+1}=\beta_{j'+1}$ when $\alpha_j=\beta_{j'}=0$ by Lemma \ref{lm:consecutive}.
If $\alpha_2=0$, then $\beta_2=0$ and $\alpha_3=\beta_3$. As $s(\alpha)=s(\beta)$, $\alpha_4=\beta_4$,   then $\alpha$ and $\beta$ are equivalent.
If $\alpha_3=0$, then $ \alpha_2=\beta_2\ne 0$ and $\beta_3=0$, $\alpha_4=\beta_4$ by Lemma \ref{lm:consecutive}. Hence $\alpha $ and $\beta$ are equivalent.
If $\alpha_4=0$, a similar argument works to show that $\alpha$ and $\beta$ are equivalent.
This completes the proof for Case $n=4$.

\end{proof}

\subsection{Case $n=5$}
In this section, we show that Conjecture \ref{conj:main}
holds for $n=5.$

\begin{prop}\label{n=5}
Suppose that $\omega^{-\alpha}$ and $\omega^{-\beta}$ are regular characters of $\BF_{p^5}^\times.$ 
If $S(\omega^{-(\alpha+\hat k)})=S(\omega^{-(\beta+\hat k)})$ for all  $0\leq k\leq p-1$, 
then
$\alpha$ and $\beta$ are equivalent.
\end{prop}
\begin{proof}
By Lemma \ref{lm:multi-set-6},  we have $\{\alpha_i\mid 1\leq i\leq 5\}=\{\beta_i\mid 1\leq i\leq 5\}$ as multi-sets and
we may assume that
$$
\alpha\equiv(a, \alpha_2,\alpha_3, \alpha_4,\alpha_5)\text{  and }\beta\equiv(a, \beta_2,\beta_3, \beta_4, \beta_5);$$ $\|\alpha\|_{r(\alpha)}=\|\beta\|_{r(\beta)}=0$;  $m_1(\alpha)\geq m_{r(\alpha)}(\alpha)$ and $\alpha_5,\ \beta_5<a=\|\alpha\|_1=\|\beta\|_1$.
Let $\{\alpha_i\ |\ 2\le i\le 5\}=\{ b\ge c\geq d\geq 0\}$  and $\{\beta_i\ |\ 2\le i\le 5\}=\{b'\geq c'\geq d'\ge 0\}.$
Since $\omega^{-\alpha}$ and $\omega^{-\beta}$ are regular, we have $r(\alpha),\ r(\beta)>1$.

If $a=1$, then $m(\alpha)=m(\beta)\in\{(4,1),(3,2)\}$ and  $\alpha_5=\beta_5=0$ under the assumption. 
For $m(\alpha)=m(\beta)=(4,1)$, $\alpha\equiv \beta\equiv(1,1,1,1,0)$.
For $m(\alpha)=m(\beta)=(3,2)$,   
it suffices to show that $(1,1,1,0,0)$ and $(1,1,0,1,0)$ are distinguishable. 
By Lemma \ref{Usame}, $V_1(1+p+p^2)/V_1(1+p+p^3)\equiv (p+1)!'/p!'\equiv p+1\ne 1\bmod{p^2}$. Hence $1+p+p^2$ and $1+p+p^3$ are distinguishable, which implies $\alpha\equiv\beta$.

In the rest of the proof, we assume that $a>1$ and then $p>2$.
We consider the following 4 cases.

If $m_1=4$, then $\alpha\equiv \beta\equiv(a,a,a,a,0)$.

If $m_1=3$, we only need to consider $(a,a,a,b,0)$, $(a,a,a,0,b)$, $\tau\equiv(a,a,b,a,0)$, and $\eta\equiv(a,a,0,a,b)$, where $a>b\ge 0$.
$(a,a,a,b,0)$ and $(a,a,a,0,b)$ are distinguishable from $\tau$ and $\eta$ as they are max-consecutive.
By Lemma \ref{lm:consecutive}, $(a,a,a,b,0)$ and $(a,a,a,0,b)$ are distinguishable when $b\ne 0$.
Next, we prove that $\tau$ and $\eta$ are distinguishable for $b\ne 0$.
By $p-1\geq a>b>0$, then $0<k=p-b-1<p-1$,
$\tau+\hat{k}\equiv(a-b-1, a-b,0,a-b, p-b)$ and $\eta+\hat k\equiv(a-b, a-b,p-b,a-b-1,0)$.
If $b<a-1$,
by $$\frac{V_1(\tau+\hat k)}{V_1(\eta+\hat k)}\equiv\frac{a-b}{a-b+(a-b)p}\ne 1 \bmod{p^2},$$
$\tau$ and $\eta$ are also distinguishable by Lemma \ref{Usame}.
If $b=a-1$, then $(\eta+\hat{k})$ is mini-consecutive 
but $(\tau+\hat{k})$ is not.
Hence they are distinguishable.

\underline{Case $r(\alpha)=3$ and $m_1\leq 2$.} 
Assume that $(m_1, m_2,m_3)=(2,2,1).$
By Lemmas \ref{r1} and \ref{lm:consecutive}, we consider the following distinguishable sets
   $$\{ (a,a,b, b,0)\},\  \{(a,a,b,0,b)\},\
\{\pi\equiv(a,a, 0, b,b ),\ \rho\equiv(a,b,a, 0,b)\}, $$
and  $\{\tau\equiv(a,b, 0, a,b),\ \eta\equiv(a,b, b, a,0 )\},$
where $p-1\ge a>b>0.$
For $k=p-a$, 
if $b\ne a-1$, we have
$\pi+\hat k\equiv(0,1, p-a+1, p-a+b,p-a+b );$ 
$\rho+\hat k \equiv(0,p-a+b+1,0, p-a+1,p-a+b)$;
$\tau+\hat k\equiv(0,p-a+b+1,p-a,0,p-a+b+1),$ and $\eta+\hat k\equiv(0,p-a+b+1, p-a+b, 0,p-a+1)$.
Then $\pi$ and $\rho$ (resp. $\tau$ and $\eta$) are distinguishable as the entries of $\pi+\hat k$ and $\rho+\hat{k}$ (resp. $\tau+\hat k$ and $\eta+\hat k$) are different as multi-sets.
 If $b=a-1,$ then 
 $\pi+\hat k\equiv(0,1, p-a+1, p-1,p-1);$ 
$\rho+\hat k \equiv(0,0,1, p-a+1,p-1)$;
$\tau+\hat k\equiv(1,0,p-a+1,0,0)$ and $\eta+\hat k\equiv(0,0, 0, 1,p-a+1).$
Then $\pi$ is  distinguishable from $\rho$ as the entries of $\pi+\hat k,$ and $\rho+\hat{k}$  are different as multi-sets;
$\tau$ is  distinguishable from $\eta$ as
 $ \eta+\hat{k} $ is mini-consecutive but $ \tau+\hat{k} $ is not. 

For $(m_1, m_2,m_3)=(2,1,2),$ we only need to consider the following distinguishable sets:
$$\{(a,a,b,0,0)\},\  \{(a,a,0,b,0)\},\  \{(a,a,0,0,b)\},\  \{(a, b,a,0, 0)\} $$
 $\{\tau\equiv(a,b, 0,a, 0), \eta\equiv(a,0, b,a, 0)\},$ where $p-1\ge a>b>0.$
For $k=p-b-1$,
 $\tau+\hat k\equiv(a-b-1,0, p-b,a-b-1,p-b)$ and $\eta+\hat k\equiv(a-b-1,p-b, p-1,a-b-1, p-b)$.
Then $\tau$ and $\eta$ are distinguishable as the entries of $\tau+\hat k$ and $\eta+\hat k$ are not equal as multi-sets.

 For $(m_1, m_2,m_3)=(1,3,1),$ there are the following cases:
   $$ (a,b,b,b,0) ,\   \tau\equiv(a,b,b,0,b),\ \eta\equiv(a,b,0,b,b),\  (a, 0,b,b, b). $$
By Lemma \ref{lm:consecutive}, $(a,b,b,b,0)$ (resp. $(a, 0,b,b, b)$) is distinguishable. 
For $k=p-b$,
$\tau+\hat k\equiv(a-b+1,1,1,p-b+1,0)$ and $\eta+\hat k\equiv(a-b+1,1,p-b+1,0,1)$.
Then $\tau$ and $\eta$ are distinguishable by Lemma \ref{lm:consecutive}.
 
 \underline{Case $r(\alpha)=4$  and $m_1\leq 2$.}
 For  $(m_1, m_2, m_3, m_4)=(2,1,1,1),$
if $\alpha_1=\alpha_2=a,$ then $\alpha$ is distinguishable by Lemma \ref{lm:consecutive},
with the similar argument to the case $m_1=1$ when $n=4$.
By Lemma \ref{lm:consecutive}, the rest of distinguishable sets are
$$\{\eta_1\equiv(a,b,a,c,0), \eta_2\equiv(a,c,a,b,0) , \eta_3\equiv(a,0,a,b,c),  
 \eta_4\equiv(a, 0, a, c, b)\},$$
$$ \{ (a,c,a,0,b)\},\  \{(a, b,a,0, c)\},$$
where   $p-1\ge a> b> c>0.$ Let $w=p-b-1.$
We have the $p$-adic expansions:
 \begin{eqnarray*}
 &&  \eta_1+\hat w\equiv(a-b-1,0,a-b,p-b+c,p-b-1),\\
 &&\eta_2+\hat w\equiv(a-b-1,p-b+c,a-b-1,0,p-b) ,\\
&& \eta_3+\hat w\equiv(a-b-1,p-b,a-b-1,0,p-b+c),\\
&&\eta_4+\hat w\equiv(a-b-1, p-b, a-b-1, p-b+c, p-1),
\end{eqnarray*}
and we obtain distinguishable sets
$ \{\eta_1\},\ \{\eta_4\}\text{ and }\{\eta_2,\ \eta_3\}$ by comparing those multi-sets formed by their entries.

If  $b\ne a-1,$ then
$\eta_2$ and $\eta_3$
 are distinguishable by Lemma \ref{lm:consecutive}.
 If $b= a-1,$
\begin{eqnarray*}
&&\eta_2+\widehat w\equiv(0,p-b+c,0,0,p-b) \\ &&\eta_3+\widehat w\equiv(0,p-b,0,0,p-a+c) 
\end{eqnarray*}
are distinguishable by reducing to the case $m_1(-(\eta_2+\widehat{p-a}))=3$ as previous discussed cases.

For $(m_1, m_2, m_3, m_4)=(1,2,1,1),$
by Lemma \ref{r1}, we have the following distinguishable sets,
$$\{\eta\equiv(a,b,c,0,b),\ \eta'\equiv(a,b,0,b,c)\},\
\{\tau\equiv(a,c,b,0,b), \ \tau'\equiv(a,c,0,b,b)\},$$
$$\{ \pi \equiv(a,0,b,b,c), \ \pi'\equiv(a, 0,b, c, b)\}, \
 \{ \rho\equiv(a,b,b,c, 0),\  \rho'\equiv(a,b,c,b,0)\}, $$
 $$\{(a,c,b,b,0)\},  \
\text{ and } \{\delta\equiv (a,b,b,0,c),\ \delta'\equiv(a,b,0,c,b)\},\ \{ (a, 0, c,b,b)\},$$
where   $p-1\ge a> b> c>0.$  Let $z=p-b$.

Since $$\delta+\hat z\equiv(a-b,1,1,p-b+1,p-b+c),\ \delta'+\hat z\equiv(a-b+1,1,p-b+1,p-b+c,0),$$
they are distinguishable by Lemma \ref{lm:multi-set-6}.

If $c\ne b-1,$ then
$$\eta+\hat z\equiv(a-b+1,1,p-b+c+1,p-b,0),\ \eta'+\hat z\equiv(a-b,1,p-b+1,0,p-b+c+1);$$
$$\tau+\hat z\equiv(a-b+1,p-b+c+1,0,p-b+1,0),\ \tau'+\hat z\equiv(a-b+1,p-b+c+1,p-b,0,1);$$
$$\pi+\hat z \equiv(a-b ,p-b+1,0,1 ,p-b+c+1 ),\ \ \pi'+\hat z\equiv(a-b+1,p-b+1,0, p-b+c+1, 0);$$
$$\rho+\hat z\equiv(a-b,1,1,p-b+c+1, p-b),\ \rho'+\hat z\equiv(a-b,1,p-b+c+1,0,p-b+1).$$
Then they are all distinguishable by Lemma \ref{lm:multi-set-6}.
 
If $b=c+1,$ then
$$\eta+\hat z\equiv(a-b+1,1,0,p-b+1,0),\ \eta'+\hat z\equiv(a-b+1,1,p-b+1,0,0);$$
$$\tau+\hat z\equiv(a-b+1,0,1,p-b+1,0),\  \tau'+\hat z\equiv(a-b+1,0,p-b+1,0,1);$$
$$\pi+\hat z \equiv(a-b+1 ,p-b+1,0,1 ,0 ),\ \pi'+\hat z\equiv(a-b+1,p-b+1,0,0, 1);$$
$$\rho+\hat z\equiv(a-b,1,1,0, p-b+1),\ \rho'+\hat z\equiv(a-b,1,0,1,p-b+1).$$
Then $\eta$ and $\eta'$ (resp. $\rho$ and $\rho'$)  are distinguishable  by Lemma \ref{lm:consecutive};
 $\tau$ and $\tau' $ are distinguishable by
$V_1(\tau+\hat z)\ne V_1(\tau'+\hat z);$
 $\pi$ and $\pi'$  are distinguishable by Lemma \ref{lm:consecutive} for $a\ne b+1$ and by reducing to the case $m_1(\pi)=3$ for $a=b+1.$

The case $(m_1, m_2, m_3, m_4)=(1,1,2,1)$ is symmetric to the previous one.

\underline{Case $r(\alpha)=5$.}
Assume that $\alpha_1=\beta_1=a,$ and \
$$\{\alpha_i\ | \ 1\le i\le 5\}=\{\beta_i\ | \ 1\le i\le 5\}=\{a,b,c,d,0\},$$
where $p-1\ge a>b>c>d>0.$ First, we have $\alpha_2=\beta_2$ by Lemma \ref{lm:consecutive}.

For $\alpha_2\ne b-1$ and $\alpha_2\ne a-1,$ applying Lemma \ref{lm:consecutive} and considering the entry next to the unique $0$ entry in
 $\alpha+\widehat{p-b}$ (resp. $\beta+\widehat{p-b}),$ we see that the entries next to $b$ in $\alpha$ and $\beta$ are the same. Hence $\alpha=\beta$ by Lemma \ref{lm:consecutive}.

If  $\alpha_2= a-1,$ then $b=a-1$. By considering the entry next to the unique $0$ entry in
 $\alpha+\widehat{p-b-1}$ (resp. $\beta+\widehat{p-b-1}),$ we see that the entries next to $b$ in $\alpha$ and $\beta$ are the same. Hence $\alpha=\beta$ by Lemma \ref{lm:consecutive}.

For $\alpha_2=b-1$,then $c=b-1$ and we have the following distinguishable sets:
\begin{eqnarray*}
&&\{\eta\equiv(a,c,b, d, 0),\eta'\equiv(a,c,d,b,0)\},\ \{\tau\equiv(a,c,0,b,d), \tau'\equiv (a,c,d,0,b)\},\\
&&\{\pi\equiv(a,c,0,d,b),\pi'\equiv (a,c,b,0,d)\}.\end{eqnarray*}
Since $\CG_b(\eta)^\circ$ (resp. $\CG_b(\tau)^\circ$) and $\CG_b(\eta')^\circ$ (resp. $\CG_b(\tau')^\circ$) have the different numbers of connected components, we have $\eta,\ \eta',$ $\tau$ and $\tau$ are distinguishable.
Also $\pi$ and $\pi'$ are distinguishable by Lemma \ref{lm:multi-set-6}. 

\end{proof}

\subsection{Case $p=2$ and Mersenne primes} \label{sec:example-2}
In this section, we will consider the case when $2^n-1$ is prime and $p=2$.
Recall that {\bf Mersenne primes} are primes of the form $2^n-1$, and in this case $n$ must be a prime.
We will verify that Conjecture \ref{conj:main} holds in this special case.
\begin{prop}\label{prop:2-Mersenne}
If $2^n-1$ is a prime (Mersenne prime), Conjecture \ref{conj:main} holds.
\end{prop}
\begin{proof}
We will prove this statement by considering the factorization of the principal ideals generated by the Gauss sums in cyclotomic fields. Let us recall Theorem 11.2.2 in \cite{BEW98}. We follow the notation in \cite[Section 11.2]{BEW98}.
Define the cyclotomic field $M=\BQ(\varepsilon, \mu_{p^n-1}),$ where $\varepsilon$ is a primitive $p$-th root of unity. 
Let $\FO_M$ be the ring of integers of $M$. 
Denote by $T$ a complete set of coset representatives for the multiplicative quotient group $R/\langle p \rangle$ where $R=(\BZ/(p^n-1)\BZ)^\times$,
 so the cardinality of $T$ is $ \phi(p^n-1)/ n$ where $\phi$ is the Euler $\phi$-function.

For any integer $j$ coprime to $p^n-1$, define the automorphism $\sig_{j}$ of $M$
by
 $$\sig_{j}:\mu_{p^n-1}\mapsto \mu_{p^n-1}^j, \sig_j\colon \varepsilon\mapsto \varepsilon.$$
Write $\FP$ to be a prime ideal lying above $p$ in $\BQ(\varepsilon,\mu_{p^n-1})$
and $\FP_j=\sig_j(\FP)$. 
Then we obtain the prime ideal factorization
$p\FO_M=\prod_{j\in T}\FP_j^{p-1}.$
By  \cite[Theorem 11.2.2]{BEW98}, for an integer $a\not\equiv 0 \bmod{p^n-1},$
\begin{equation}\label{eq:factorization}
S(\omega^{-a})\FO_M=\prod_{j\in T}\FP_{j^{-1}}^{s(aj)}, 	
\end{equation}
where $j^{-1}$ is the inverse of $j \bmod{p^n-1}$.


Now let us apply the ideal factorization Eq. \eqref{eq:factorization} to our case.
Assume that $S(\omega^{-\alpha})=S(\omega^{-\beta})$ for non-trivial characters $\omega^{-\alpha}$ and $\omega^{-\beta}$.
Since $2^n-1$ is a prime, the orders of $\omega^{-\alpha}$ and $\omega^{-\beta}$ are $2^n-1$.
If $S(\omega^{-\alpha})=S(\omega^{-\beta})$, then $s(c\alpha)=s(c\beta)$ for all $c$ coprime to $2^n-1$ by  Eq. \eqref{eq:factorization}.

Take an integer $c$ such that $c\equiv \alpha^{-1}\bmod{2^n-1}$. 
Then $s(1)=s(\beta\alpha^{-1})$.
Note that $s(c)=1$ if and only if $c=2^i$ for some $i$.
Then $\beta\alpha^{-1}\equiv 2^i\bmod{2^n-1}$ is equivalent to $\beta\equiv 2^i\alpha\bmod{2^n-1}$.
This completes the proof.
\end{proof}

\begin{exmp}
The 50-th Mersenne prime, $M_{50}=2^{77232917}-1$, was found at December 26, 2017, and it becomes the largest prime number known to the mankind. 
Then Conjecture \ref{conj:main} holds for $\GL_{77232917}(\BF_2)$.
\end{exmp}

\subsection{$n\times 1$ counterexamples} \label{sec:example-3}
In this section, we give  counterexamples to $n\times 1$ Local Converse Problem, i.e. we need more than twisted gamma factors against characters of $\BF_p^\times$ to determine a unique cuspidal representation of $\GL_n(\BF_p)$ up to isomorphism.
Note that these examples do not contradict to Conjecture \ref{conj2} or \ref{conj:main}, since these representations are not primitive.

Let $t\geq 3$ be an integer and $p=3$. 
Referring to \cite[Theorem 11.6.1]{BEW98}, for all characters $\chi$ of $\BF_{3^{2t}}^\times$ of order $3^t+1$,
one has $S(\chi)=-p^{t}$.
For instance, let $a$ be a positive integer coprime to $p^t+1$.
Taking $\chi=\omega^{a(p^{t}-1)}$,
we have a regular character $\chi$.
For $k=0$ or $1$, by definition $a(p^{t}-1)+\hat{k}=a(p^{t}-1)+k\frac{p^{t}+1}{2}(p^t-1)=(a+k\frac{p^{t}+1}{2})(p^t-1)$,
it follows that $\omega^{a(p^{t}-1)+\hat{k}}$ is of order $p^{t}+1$ or $\frac{p^{t}+1}{2}$ depending on the parity of $t$.
In both cases, $S(\omega^{a(p^{t}-1)+\hat{k}})=(-p)^{t}$.
If the Euler $\phi$-function $\phi(p^t+1)\geq 4t$, then there exist two positive integers $a$ and $b$ such that $a$ and $b$ are coprime to $p^t+1$ and $a\neq p^jb\bmod{p^{2t}-1}$ for all $j$.
Then we have two regular characters $\omega^{\alpha}$ and $\omega^{\beta}$ of order $p^t+1$ such that $\alpha\neq p^j\beta\bmod{p^{2t}-1}$ for all $j$, where $\alpha=a(p^{t}-1)$ and $\beta(p^{t}-1)$.
Let $\pi$ and $\tau$ be the irreducible cuspidal representations of $\GL_{2t}(\BF_p)$ corresponding to $\omega^{\alpha}$ and $\omega^{\beta}$ respectively.
By $\alpha\neq p^j\beta\bmod{p^{2t}-1}$ for all $j$, $\pi$ and $\tau$ are not isomorphic.
However, for all characters $\eta$ of $\BF_p^\times$, $\gamma(\pi\times\eta,\psi)=\gamma(\pi\times\eta,\psi)$.
For example, when $n=6$, $\phi(3^3+1)=12$ and this is a counter example for $\GL_6(\BF_3)\times \GL_1(\BF_3)$ Local Converse Problem.

\subsection{  $n\times 1$ Local Converse Theorem for level zero supercuspidal}
Following Theorem \ref{level0} and the established $n\times 1$ Local Converse Theorem for finite field case,    the following theorems hold.
\begin{thm}\label{nx1padic}
Let $\CF_p$ be the   $p$-adic field with prime residue field.
Let $\pi_1$ and $\pi_2$ be a pair of irreducible, unitarizable,  level zero supercuspidal representations of $\GL_n(\CF_p),\ 2\le n\le 5.$
Suppose that $$\Gamma(s, \pi_1\times \chi, \psi_{\CF_p} )= \Gamma(s, \pi_2\times \chi, \psi_{\CF_p}), \text{ for all }\chi\in \widehat \CF_p^\times,$$
then $\pi_1$ and $\pi_2$ are isomorphic.
\end{thm}
\begin{proof}
The result follows \cite{JiNiS15} for $n=2,\ 3$ and Theorem \ref{level0}; \ref{n=4}; \ref{n=5} for $n=4,\ 5.$

\end{proof}

\begin{thm}\label{Yunpadic}
Let $\CF$ be a $p$-adic field with $q$ elements in the residue field.
Let $\pi_1$ and $\pi_2$ be a pair of irreducible, unitarizable,  level zero supercuspidal representations of $\GL_n(\CF),$ $n<\frac{q-1}{2\sqrt{q}}+1$.
Suppose that $$\Gamma(s, \pi_1\times \chi, \psi_{\CF_p} )= \Gamma(s, \pi_2\times \chi, \psi_{\CF}), \text{ for all }\chi\in \widehat \CF^\times,$$
then $\pi_1$ and $\pi_2$ are isomorphic.
\end{thm}
\begin{proof}
The result follows \cite{JiNiS15} for $n=2,\ 3$, Theorems \ref{level0} and \ref{th:G}. 
\end{proof}

\begin{thm}\label{Mersennepadic}
Let $\CF_2$ be the the $2$-adic field with residue field of 2 elements. Assume that  $2^n-1$ is a Mersenne prime.
Let $\pi_1$ and $\pi_2$ be a pair of irreducible, unitarizable,  level zero supercuspidal representations of $\GL_n(\CF_2).$
Suppose that $$\Gamma(s, \pi_1\times \chi, \psi_{\CF_2} )= \Gamma(s, \pi_2\times \chi, \psi_{\CF_2} ), \text{ for all }\chi\in \widehat \CF_2^\times,$$
then $\pi_1$ and $\pi_2$ are isomorphic.
\end{thm}

\newpage

  \appendix

\section{A remark on Gauss sums}
 







\bigskip

\centerline{Zhiwei Yun \footnote{Supported by the Packard Foundation.}} 
\bigskip

\begin{thm}\label{th:G} Let $q$ be a prime power, $\psi: \FF_{q}\to \CC^{\times}$ a nontrivial character. Let $n\ge1$ be an integer satisfying $n<\frac{q-1}{2\sqrt{q}}+1$. Let $\chi_{1},\chi_{2}:\FF_{q^{n}}^{\times}\to \CC^{\times}$ two characters. Suppose that for any character $\eta:\FF^{\times}_{q}\to \CC^{\times}$ there is an equality of Gauss sums (for the finite field $\FF_{q^{n}}$)
\begin{equation*}
G(\chi_{1}\cdot \eta, \psi)=G(\chi_{2}\cdot \eta, \psi).
\end{equation*}
Then $\chi_{1}$ and $\chi_{2}$ are in the same Frobenius orbit, i.e., there exists $j\in\ZZ$ such that $\chi_{1}=\chi_{2}^{q^{j}}$. 
\end{thm}

Before giving the proof we recall some results of Deligne and Katz on Kloosterman sheaves. Let $\ell$ be any prime different from $\textup{char}(\FF_{q})$, and choose an embedding $\iota: \Qlbar\incl \CC$. The characters $\psi,\chi_{i}$ all take values in $\Qlbar\subset\CC$. The embedding $\iota$ also gives an archimedean norm $|\cdot |$ on $\Qlbar$.

Consider the diagram of schemes over a finite field $k$
\begin{equation*}
\xymatrix{   \Gm & \Gm^{n}\ar[r]^-{\s}\ar[l]_-{\pi} & \AA^{1}}
\end{equation*}
where $\s(a_{1},\cdots, a_{n})=a_{1}+\cdots+a_{n}$ and $\pi(a_{1},\cdots,a_{n})=a_{1}\cdots a_{n}$. Let $\xi_{1},\cdots, \xi_{n}$ be characters of $k^{\times}$, each giving a Kummer local system $\cL_{\xi_{i}}$ on $\Gm$. The additive character $\psi$ gives an Artin-Schreier local system $\AS_{\psi}$ on $\AA^{1}$. We form the $\ell$-adic complex on $\Gm$
\begin{equation*}
\Kl(\psi; \xi_{1},\cdots, \xi_{n}):=\bR\pi_{!}\s^{*}\AS_{\psi}[n-1].
\end{equation*}
We shall need the following facts about $\Kl(\psi; \xi_{1},\cdots, \xi_{n})$.

\begin{thm}[Deligne and Katz]\label{th:Kl}
\begin{enumerate}
\item\label{loc} $\Kl(\psi; \xi_{1},\cdots, \xi_{n})$ is a local system of rank $n$ concentrated in degree 0 (and is called the {\em Kloosterman sheaf}). 
\item\label{irr} $\Kl(\psi; \xi_{1},\cdots, \xi_{n})$ is irreducible as a $\Qlbar$-local system over $\Gm\ot\kbar$ (in fact already as a representation of the inertia group at $\infty$).
\item\label{pure} $\Kl(\psi; \xi_{1},\cdots, \xi_{n})$ is pure of weight $n-1$ on $\Gm$ with respect to the embedding $\iota$.
\item\label{dual} The dual local system of  $\Kl(\psi; \xi_{1},\cdots, \xi_{n})$ is $\Kl(\ov\psi; \xi^{-1}_{1},\cdots, \xi^{-1}_{n})\ot\Qlbar(n-1)$, where $\ov\psi(a)=\psi(-a)$.
\item\label{m0} The semi-simplified monodromy of $\Kl(\psi; \xi_{1},\cdots, \xi_{n})$ at $0$,  as a representation of the tame inertia group at $0$, is the direct sum of characters given by $\xi_{1},\cdots, \xi_{n}$. 
\item\label{tensor} For any two sets of multiplicative characters $\xi_{1},\cdots,\xi_{n}$ and $\eta_{1},\cdots, \eta_{n}$ of $k^{\times}$, the Swan conductor of $\Kl(\psi; \xi_{1},\cdots, \xi_{n})\ot\Kl(\ov\psi; \eta_{1},\cdots, \eta_{n})$ at $\infty$ is $n-1$.
\end{enumerate}
\end{thm}
For \eqref{loc} and \eqref{pure}, see \cite[Thm. 7.8 and Remark 7.18]{Deligne}; for \eqref{irr}, see \cite[Cor. 4.1.2]{Katz}; for  \eqref{dual}, see \cite[Cor.4.1.3]{Katz}; for \eqref{m0}, see \cite[Prop 7.3.2]{Katz}; for \eqref{tensor}, see \cite[Prop.10.4.1]{Katz}.

\begin{proof}[Proof of Theorem \ref{th:G}] Let $k=\FF_{q}$ and $k'=\FF_{q^{n}}$. Consider the diagram of schemes over $k=\FF_{q}$
\begin{equation*}
\xymatrix{   \Gm & \Res_{k'/k}\Gm\ar[r]^-{\Tr}\ar[l]_-{\Nm} & \AA^{1}}.
\end{equation*}
Recall that $\chi_{i}$ gives a Kummer local system $\cL_{\chi_{i}}$ on $\Res_{k'/k}\Gm$: the Lang cover $\pi: \Res_{k'/k}\Gm\to \Res_{k'/k}\Gm$ is a $(\Res_{k'/k}\Gm)(k)=k'^{\times}$-torsor, and $\cL_{\chi_{i}}$ is the direct summand of $\pi_{*}\Qlbar$ on which $k'^{\times}$ acts via $\chi_{i}$.  Consider the $\ell$-adic complex on $\Gm$
\begin{equation*}
\cK_{i}=\bR\Nm_{!}(\Tr^{*}\AS_{\psi}\otimes \cL_{\chi_{i}}).
\end{equation*}
For $a\in k^{\times}$, consider the trace of the geometric Frobenius at $a$ acting on the stalk of $\cK_{i}$:
\begin{equation*}
\cK_{i}(a):=\Tr(\Frob_{a}, \cK_{i}|_{a}).
\end{equation*}
From the construction of $\cK_{i}$ and the Grothendieck-Lefschetz trace formula, we have for any $a\in k^{\times}=\Gm(k)$,
\begin{equation*}
\cK_{i}(a)=(-1)^{n-1}\sum_{b\in k'^{\times}, \Nm(b)=a}\chi_{i}(b)\psi(\Tr_{k'/k}(b)).
\end{equation*}

The fact that $G(\chi_{1}\cdot\eta, \psi)=G(\chi_{2}\cdot\eta, \psi)$ for all characters $\eta$ of $\FF^{\times}_{p}$ implies, via the inverse Fourier transform on $k^{\times}$, that
\begin{equation*}
\sum_{b\in k'^{\times}, \Nm(b)=a}\chi_{1}(b)\psi(\Tr_{k'/k}(b))=\sum_{b\in k'^{\times}, \Nm(b)=a}\chi_{2}(b)\psi(\Tr_{k'/k}(b))
\end{equation*}
for all $a\in k^{\times}$. Therefore,
\begin{equation}\label{K1=K2}
\cK_{1}(a)=\cK_{2}(a), \quad \forall a\in k^{\times}.
\end{equation}

If we base change from $k$ to $k'$, then $ \Res_{k'/k}\Gm$ becomes a split torus $\Gm^{n}$, and $\cL_{\chi_{i}}$ becomes the external tensor product of the Kummer local systems $\cL_{\chi^{q^{j}}_{i}}$ for $j\in\ZZ/n\ZZ\cong\Gal(k'/k)$. Therefore, 
\begin{equation*}
\cK_{i}|_{\Gm\ot k'}\cong \Kl(\psi, \chi_{i},\chi^{q}_{i},\cdots, \chi^{q^{n-1}}_{i}).
\end{equation*}

Suppose $\chi_{1}$ and $\chi_{2}$ are not in the same Frobenius orbit, then the multisets of characters $\{\chi_{1},\chi^{q}_{1},\cdots, \chi^{q^{n-1}}_{1}\}$ and $\{\chi_{2},\chi^{q}_{2},\cdots, \chi^{q^{n-1}}_{2}\}$ are different.
Therefore $\cK_{1}$ and $\cK_{2}$ have different semi-simplified monodromy at $0$ by Theorem \ref{th:Kl}\eqref{m0}. 
Hence they are not isomorphic to each other over $\Gm\ot\kbar$.

Now consider the cohomology of  $\cK^{\vee}_{1}\otimes\cK_{2}$. Since $\cK_{1}$ and $\cK_{2}$ are not geometrically isomorphic, 
one has $\cohog{*}{\Gm\ot\kbar,\cK^{\vee}_{1}\ot\cK_{2}}=0$ for $*\ne1$. 
By Theorem \ref{th:Kl} \eqref{tensor} and \eqref{dual}, we have
\begin{equation*}
\Swan_{\infty}(\cK^{\vee}_{1}\ot\cK_{2})=n-1.
\end{equation*}
Therefore by the Grothendieck-Ogg-Shafarevich formula, 
\begin{equation}\label{rk12}
\dim \cohoc{1}{\Gm\ot\kbar,\cK^{\vee}_{1}\ot\cK_{2}}=\Swan_{\infty}(\cK^{\vee}_{1}\ot\cK_{2})=n-1.
\end{equation}
By the Grothendieck-Lefschetz trace formula,
\begin{equation}\label{K12}
-\Tr(\Frob,\cohoc{1}{\Gm\ot\kbar,\cK^{\vee}_{1}\ot\cK_{2}})=\sum_{a\in k^{\times}}\cK^{\vee}_{1}(a)\cK_{2}(a).
\end{equation}

One the other hand, consider the cohomology of $\uEnd(\cK_{1})=\cK^{\vee}_{1}\otimes\cK_{1}$.  Let $\uEnd^{0}(\cK_{1})$ be the trace-free summand of $\uEnd(\cK_{1})$, then $\uEnd(\cK_{1})\cong\Qlbar\oplus \uEnd^{0}(\cK_{1})$. Since $\cK_{1}$ is geometrically irreducible by Theorem \ref{th:Kl} \eqref{irr}, $\cohoc{*}{\Gm\ot\kbar,\uEnd^{0}(\cK_{1})}=0$ for $*\ne1$. The same Swan conductor calculation as for $\cK_{1}\ot\cK_{2}^{\vee}$ shows that
\begin{equation}\label{rkEnd}
\dim \cohoc{1}{\Gm\ot\kbar,\uEnd^{0}(\cK_{1})}=\Swan_{\infty}(\uEnd^{0}(\cK_{1}))=\Swan_{\infty}(\cK^{\vee}_{1}\ot\cK_{1})=n-1.
\end{equation}
By the Grothendieck-Lefschetz trace formula,
\begin{eqnarray}
\notag&&-\Tr(\Frob,\cohoc{1}{\Gm\ot\kbar,\End^{0}(\cK_{1})})+q-1\\
\notag&=&-\Tr(\Frob,\cohoc{1}{\Gm\ot\kbar,\End^{0}(\cK_{1})})+\Tr(\Frob, \cohoc{*}{\Gm\ot\kbar,\Qlbar})\\
\notag&=&\Tr(\Frob,\cohoc{*}{\Gm\ot\kbar,\cK^{\vee}_{1}\ot\cK_{1}})\\
\label{K11}&=&\sum_{a\in k^{\times}}\cK^{\vee}_{1}(a)\cK_{1}(a).
\end{eqnarray}
Comparing \eqref{K11} and \eqref{K12}, using \eqref{K1=K2}, we get
\begin{equation}\label{tr eq}
-\Tr(\Frob,\cohoc{1}{\Gm\ot\kbar,\cK^{\vee}_{1}\ot\cK_{2}})=-\Tr(\Frob,\cohoc{1}{\Gm\ot\kbar,\End^{0}(\cK_{1})})+q-1.
\end{equation}
Now since $\cK_{i}$ are pure of weight $n-1$ by Theorem \ref{th:Kl} \eqref{pure}, we have $\cK^{\vee}_{1}\ot\cK_{2}$ and $\End^{0}(\cK_{1})$ are both pure of weight 0. 
Hence $\cohoc{1}{\Gm\ot\kbar,\cK^{\vee}_{1}\otimes\cK_{2}}$ and $\cohoc{1}{\Gm\ot\kbar,\uEnd^{0}(\cK_{1})}$ have weights $\le 1$ (in fact pure of weight $1$ but we shall not need this fact). 
Using the rank calculations \eqref{rk12} and \eqref{rkEnd}, we get
\begin{eqnarray*}
|\Tr(\Frob,\cohoc{1}{\Gm\ot\kbar,\cK^{\vee}_{1}\ot\cK_{2}})|\le (n-1)\sqrt{q},\\
|\Tr(\Frob,\cohoc{1}{\Gm\ot\kbar,\End^{0}(\cK_{1})})|\le (n-1)\sqrt{q}.
\end{eqnarray*}
Together with \eqref{tr eq} we get
\begin{equation*}
q-1\le 2(n-1)\sqrt{q}
\end{equation*}
or $n\ge \frac{q-1}{2\sqrt{q}}+1$ equivalently, contradicting the condition that $n<\frac{q-1}{2\sqrt{q}}+1$. Therefore $\chi_{1}$ and $\chi_{2}$ must be in the same Frobenius orbit. 
\end{proof}

\begin{rmk}\label{rmk:appendix} Theorem \ref{th:G} admits the following generalization. For any \'etale $\FF_{q}$-algebra $A$ of  degree $n$, any character $\chi: A^{\times}\to \CC^{\times}$, one can define the Gauss sum
\begin{equation*}
G_{A}(\chi,\psi):=\sum_{a\in A^{\times}}\chi(a)\psi(\Tr_{A/k}(a)).
\end{equation*} 
Now let $A,A'$ be two  \'etale $\FF_{q}$-algebras  of  degree $n$, and $\chi,\chi'$ be characters of $A^{\times}$ and $A'^{\times}$ respectively. The character  $\chi$ can be viewed as a $\Frob$-stable divisor (i.e., integral linear combination of elements) of $\Gamma=\varinjlim \wh{\FF^{\times}_{q^{n}}}$ (using inflation via norm maps as transition maps) of degree $n$. If $\chi$ and $\chi'$ define the same divisors of $\Gamma$, then for any character $\eta$ of $\FF_{q}^{\times}$, we have
\begin{equation}\label{GAA'}
\ep_{A}G_{A}(\chi\cdot \eta,\psi)=\ep_{A'}G_{A'}(\chi'\cdot\eta,\psi).
\end{equation}
Here $\ep_{A}=(-1)^{n-r}$, if $A\cong k_{1}\times\cdots\times k_{r}$ where $k_{i}$ are fields (so $\ep_{A}$ is the sign of the permutation action of Frobenius on the $\kbar$-points of $\Spec A$).

The corresponding generalization of Theorem \ref{th:G} reads:   if $n<\frac{q-1}{2\sqrt{q}}+1$ and \eqref{GAA'} holds for all characters $\eta$ of $\FF_{q}^{\times}$, then $\chi$ and $\chi'$ define the same divisors of $\Gamma$. The proof is the same as the proof of Theorem \ref{th:G}.
\end{rmk}

\begin{rmk}
Theorem \ref{th:G} does not say anything when $n$ is large compared to $\sqrt{q}$. However, Conjecture \ref{conj:Gauss-sum}  predicts that when $n$ is prime, the conclusion of Theorem \ref{th:G}  should be true for any $q$. This suggests the following conjecture for Kloosterman sheaves.

\begin{conj}
Let $n$ be a prime. Let $\psi:\FF_{q}\to \Qlbar^{\times}$ be a nontrivial character. Let $\xi_i, \eta_i\in \Gamma=\varinjlim \wh{\FF^{\times}_{q^{n}}}$ for $1\le i\le n$. Assume that the Kloosterman sheaves $\Kl(\psi; \xi_1,\cdots,\xi_n)$ and $\Kl(\psi; \eta_1,\cdots,\eta_n)$, originally defined over $\GG_{m,\ov{\FF_{q}}}$,  descend to local systems $\cK_{1}$ and $\cK_{2}$ over $\GG_{m,\FF_q}$ (this is equivalent to saying that the divisors  $\sum_{i=1}^{n}\xi_i$ and $\sum_{i=1}^{n}\eta_i$ on $\Gamma $ are Frobenius invariant). If $\cK_{1}$ and $\cK_{2}$ have the same Frobenius traces on all $\FF_{q}$-points of $\Gm$, then $\cK_1\cong\cK_2$, and in particular,  $\sum_{i=1}^{n}\xi_i=\sum_{i=1}^{n}\eta_i$ as divisors on $\Gamma$.
\end{conj}

\end{rmk}


\end{document}